\documentclass[11pt,twoside,reqno]{amsart}
\usepackage{amsmath, amsfonts, amsthm, amssymb, graphicx, amscd,color,colortbl,}
\usepackage{enumerate}

\usepackage{graphics}
\usepackage{epsfig}
\usepackage{enumerate} \usepackage{array,cite}
\def\bbm{\begin{bmatrix}}
\def\ebm{\end{bmatrix}}
\def\bpm{\begin{pmatrix}}
\def\epm{\end{pmatrix}}
\def\bvm{\begin{vmatrix}}
\def\evm{\end{vmatrix}}
\def\bin{\begin{enumerate}}
\def\ein{\end{enumerate}}
\def\bit{\begin{itemize}}
\def\eit{\end{itemize}}
\def\bid{\begin{description}}
\def\eid{\end{description}}

\newcommand \eps{\varepsilon}

\newcommand \R{\mathbb{R}}
\newcommand \bR{\mathbb{R}}

\newcommand{\Nbhd}{{\mathcal N}}

\newcommand{\x}{\mathbf{x}}
\newcommand{\y}{\mathbf{y}}
\newcommand{\zz}{\mathbf{z}}
\newcommand{\dd}{\rm d}
\newcommand{\vv} {{\mathbf v}}

\newcommand{\xxi}{{\boldsymbol \xi}}
\newcommand{\nnu}{{\boldsymbol\nu}}
\newcommand{\ttau}{{\boldsymbol\tau}}
\newcommand{\rr}{\mathbf{r}}
\newcommand \Sonic{\Gamma_{\rm sonic}}

\newcommand \Shock{\Gamma_{\rm shock}}

\newcommand \Wedge{\Gamma_{\rm wedge}}
\newcommand \Symm{\Gamma_{\rm sym}}
\newcommand \OmegaPM{\Omega^{\rm ext}}
\newcommand{\dist}{ \mbox{\rm dist}}
\newcommand{\divg}{ \mbox{\rm div}}
\newcommand{\curl}{ \mbox{\rm curl}}

\newcommand{\grad}{ \nabla}

\newcommand{\xx}{{\mathbf x}}

\newcommand{\uu}{{\mathbf u}}

\newcommand{\thetaw}{\theta_{\rm w}}

\newcommand{\bt}{\boldsymbol{\tau}}
\newcommand{\bn}{\boldsymbol{\nu}}

\newcommand{\Gw}{\Gamma_{\text{\rm wedge}}}

\newcommand{\Gso}{\Gamma_{\text{\rm sonic}}}

\newcommand{\Gsh}{\Gamma_{\text{\rm shock}}}

\newcommand{\Nex}{{N_1}}
\newcommand{\Nint}{{N_2}}

\newcommand{\Rem}{{\mathcal R}}
\newcommand{\sign}{\mathrm{sign}}

\newcommand{\GammaInt}{\Gamma^{\mbox{\rm \footnotesize int}}}
\newcommand{\GammaExt}{\Gamma^{\mbox{\rm \footnotesize ext}}}
\newcommand{\NuExtI}{\nnu_i^{\mbox{\rm \footnotesize ext}}}
\newcommand{\TauExtI}{\ttau_i^{\mbox{\rm \footnotesize ext}}}
\newcommand{\PZer}{\mathcal{P}_0}

\newtheorem{thm}{{Theorem}}[section]

\newtheorem{lemma}[thm]{{Lemma}}
\newtheorem{theorem}[thm]{{Theorem}}

\newtheorem{definition}[thm]{Definition}

\newtheorem{remark}[thm]{Remark}

\numberwithin{equation}{section}
\begin{document}
\title[Low Regularity of the Riemann Solutions for the Isentropic Euler System]{Low Regularity of Self-Similar Solutions of
\\ Two-Dimensional Riemann Problems with Shocks\\ for the Isentropic Euler System}

\author{Gui-Qiang G. Chen}
\address{Gui-Qiang G. Chen, Mathematical Institute, University of Oxford,
Oxford, OX2 6GG, UK}
\email{gui-qiang.chen@maths.ox.ac.uk}

\author{Mikhail Feldman}
\address{Mikhail Feldman, Department of Mathematics,
         University of Wisconsin-Madison,
         Madison, WI 53706-1388, USA}
\email{feldman@math.wisc.edu}

\author{Wei Xiang}
\address{Wei Xiang, Department of Mathematics, City University of Hong Kong, Kowloon, Hong Kong, P.R. China}
\email{weixiang@cityu.edu.hk}

\date{\today}
\keywords{Low regularity, transonic shock, free boundary,
compressible flow, potential flow, self-similar,
conservation laws, shock reflection-diffraction,
regular shock reflection,
Prandtl reflection, Lighthill diffraction, nonlinear,
fine properties.\\
*Corresponding author: Gui-Qiang G. Chen}


\begin{abstract}
We are concerned with the low regularity of self-similar solutions of
two-dimensional Riemann problems for the isentropic Euler system.
We establish a general framework for the analysis of the local regularity of
such solutions for a class of two-dimensional Riemann problems for the isentropic Euler system,
which includes the regular shock reflection problem,
the Prandtl reflection problem,
the Lighthill diffraction problem, and the four-shock Riemann problem.
We prove that the velocity is not in $H^1$ in the subsonic domain for the self-similar 
solutions of these problems in general.
This indicates that the self-similar solutions of the Riemann problems
with shocks for the isentropic Euler system are of much more complicated
structure
than those for the Euler system for potential flow;
in particular, the velocity is not necessarily continuous
in the subsonic domain.
The proof is based on a regularization of the isentropic Euler system
to derive the transport equation for the vorticity,
a renormalization argument extended to the case of domains with boundary,
and DiPerna-Lions-type commutator estimates.
\end{abstract}
\maketitle

\medskip
\bigskip
\bigskip
\section{Introduction}\label{sec:introduction}
We are concerned with the regularity of self-similar solutions of the Riemann problems with shocks
for the isentropic Euler system
in a general setting, including several fundamental shock problems such as
the regular shock reflection problem, the Prandtl reflection problem,
the Lighthill diffraction problem, and the four-shock Riemann problem.
In 1860, Riemann first considered a special initial value
problem with two constant states separated at the origin for
the one-dimensional isentropic Euler system in \cite{Riemann} -- now known as the Riemann problem;
this Riemann problem has played a fundamental role in the mathematical theory of hyperbolic systems
of conservation laws,
since its solutions are building blocks and
asymptotic attractors of general global entropy solutions.
Since then, a
systematic theory of one-dimensional hyperbolic systems of conservation laws has been established;
see \cite{Bressan,BressanLiuYang,ChangHsiao,DaBOOK2016,Glimm,Lax,LiuYang} and the reference cited therein.
However, multi-dimensional Riemann problems are much more complicated and completely different
from the one-dimensional case.
Even for the two-dimensional Riemann problem with four constant states given in the four quadrants
for the Euler system,
nineteen genuinely different configurations have been identified;
see \cite{CCY1995,CCY2000,CCHLW2023,LaxLiu,ZhengAMAS2006,ZhangZheng}.
Since then, rigorous global results for the $2$-D four-quadrant Riemann problem for the Euler system
were only done by Li-Zheng in \cite{LiZheng2009,LiZheng2010}.
See also \cite{ChensQu,Serre} for the 2-D Riemann problem for Chaplygin gases.

On the other hand, the regular shock reflection problem is a different type of multi-dimensional Riemann problems --
a lateral Riemann problem that involves the wedge boundaries.
Shock reflection-diffraction phenomena were first presented by Ernst Mach \cite{Mach} in 1878,
and experimental, computational, and asymptotic analysis has shown that various patterns
of shock reflection may occur, including regular and Mach reflection;
see\cite{CFBOOK2018} and the references cited therein.

More recently, the existence of global regular
reflection solutions was established for the Euler system for potential flow
in \cite{CFANN2010,CFBOOK2018}, and further properties of these solutions were
proved in \cite{BCFIVENT2009,CFXConvexity}.
In particular, the solution has a high regularity in the subsonic domain $\Omega$
(where the solution is not a constant state);
see Figs. \ref{figure:shock refelction}--\ref{figure:shock refelction-subs}
in which density $\rho$ and velocity $\vv$ are
in $C^\alpha(\overline\Omega)\cap C^\infty(\Omega)$ for some $\alpha\in (0,1)$.
More precisely, solution $(\rho, \vv)$ is in a weighted H\"{o}lder space which implies that
$(\rho, \vv)$ is in the Sobolev space $W^{1,p}(\Omega)$ for some $p>2$.
In addition, the regularity of the curved reflected-diffracted shock was shown
to be $C^\infty$ in the interior and $C^{2,\alpha}$ up to the endpoints, which is
also expected to hold for the case of the isentropic Euler system.
See also \cite{BCFQAM2013,BCF-2,ELCPAM2008} for the Prandtl reflection-diffraction configuration,
\cite{CDXARMA2014} for the Lighthill diffraction problem,
and \cite{CCHLW2023} for the four-shock Riemann problem.

In contrast, a remarkable phenomenon was first observed by
Serre in \cite{Serre-2007} which showed by a formal calculation
that, in the case of the isentropic Euler system,
the regular shock reflection solutions develop {\it vortical singularity},
specifically that the vorticity cannot be in $L^2(\Omega)$.
This implies that the velocity cannot be in $W^{1,2}(\Omega)$, {\it i.e.}, $H^1(\Omega)$,
which is lower than the regularity
of the velocity for the case of potential flow discussed above.
This in particular allows that the velocity can be discontinuous,
which at least can not be excluded by the Sobolev embeddings.
Since the calculation is formal, it is important to find out
whether the low regularity indeed necessarily holds for the regular shock reflection solutions.

In this paper, under the natural assumptions on the regularity of the self-similar solutions
in the subsonic domain $\Omega$ near the shocks,
we rigorously prove that the velocity is not in $H^1(\Omega)$ indeed.
The argument is based on the vortical singularity calculation in \cite{Serre-2007}.
We first apply this calculation to the regularized solutions carefully
for which the calculation
can be rigorously justified; however, the additional error terms appear due to the regularization.
Then we develop
DiPerna-Lions-type commutator estimates
to control the error terms when the regularization parameter tends to zero.
With this, we then prove the lower regularity property of the self-similar
solutions, by employing renormalization argument.
These self-similar solutions include the above-mentioned shock reflection problems:
the regular shock reflection problem, the Prandtl reflection problem,
the Lighthill diffraction problem,
and the four-shock Riemann problem.

This indicates that the self-similar solutions of the Riemann problems
with shocks for the isentropic Euler system
have much more complicated structure than the corresponding solutions for the Euler
system for potential flow;
in particular, the velocity is not necessarily continuous
in $\Omega$, at least their continuity can not be obtained directly by the Sobolev embeddings.
On the other hand, this argument allows the possibility that $(\rho,\vv)\in
W^{1,p}(\Omega)$ for some or even all $p\in[1,2)$, in which case there are no shocks in $\Omega$.

This paper is organized as follows:
In \S 2, we derive the isentropic Euler system in self-similar coordinates and then present the notion of entropy solutions
in the new coordinates.
In \S 3, we first formulate a general framework for analyzing
the low regularity of entropy solutions of the Riemann problems
in Definition \ref{RiemannProblSolutStruct-1}
and then establish our main theorem of this paper, Theorem \ref{lower-reg-Th},
for the entropy solutions.
Then, in \S 4, we employ the main theorem, Theorem \ref{lower-reg-Th}, for the general framework established in \S 3
to several fundamental transonic shock problems including the regular shock reflection problem,
the Prandtl reflection problem, the Lighthill diffraction problem, and the Riemann problem with four-shock interactions
in \S 4.1--\S 4.4, respectively.
The general framework and the main theorem, as well as ideas and approaches, developed in this paper
should be useful for solving other similar low regularity problems for solutions of nonlinear partial
differential equations.

%%%%%%%%%%%%%%%%%%%%%%%%%%%%%%%%%%%%%%%%%%%%%%%%%%%%%%%%%%%%%%%%%%%%%%%%

\section{The Isentropic Euler System and
Entropy Solutions with Shocks}\label{sec:potential flow wquation and free boundary problem}

In this section, we first derive the isentropic Euler system in self-similar coordinates
and then present the notion of entropy solutions, as well as the Rankine-Hugoniot conditions
and the corresponding entropy condition
across a shock (as a free boundary), in self-similar coordinates.

\subsection{The Isentropic Euler System}\label{subsec:potential flow rh condition}
\label{Euler-Eq-subseq}

As in \cite{DaBOOK2016,Serre-2007} (see also \cite{CFBOOK2018}),
the isentropic Euler equations consist of the conservation
laws of mass and momentum:
\begin{equation}\label{isentropisEulersystem}
\begin{cases}
\partial_t\rho+\divg(\rho \uu)=0,\\[1mm]
\partial_t(\rho \uu)+\divg(\rho \uu \otimes \uu)+\nabla p=0,
\end{cases}
\end{equation}
where
$\rho$ is the density, $\uu=(u_1,u_2)$ is the velocity, and
$p$ is the pressure.
The constitutive relation between pressure $p$ and density $\rho$ is through the $\gamma$-law
relation: $p=\frac{\rho^{\gamma}}{\gamma}$ after scaling, and the adiabatic exponent $\gamma>1$ is a given constant.

If an initial-boundary value problem is invariant under
the self-similar scaling:
\begin{equation}\label{urho-self-similar}
(\uu,p,\rho)=(\uu,p,\rho)(\frac{\mathbf{x}}{t}),
\end{equation}
we introduce self-similar variables $\xxi=(\xi_1,\xi_2)=\frac{\mathbf{x}}{t}\in\mathbb{R}^2$
and the pseudo-velocity $\vv=\uu-\xxi$.
Then we obtain the isentropic Euler system for self-similar flow $(\rho, \vv)=(\rho,\vv)(\xxi)$
in the form:
\begin{align}\label{selfSimEuler}
&\divg(\rho\, \vv)+2\rho=0, \\
\label{selfSimEuler-moment}
&\divg(\rho\, \vv\otimes \vv)+ 3\rho \vv+ \grad p=0.
\end{align}

If $(\rho, \vv)\in C^1$, then we can combine the equations above to rewrite \eqref{selfSimEuler-moment} as
\begin{equation}\label{Cons-Moment-nonDiv}
(\vv\cdot \grad) \vv +\vv+\grad h(\rho)=0,
\end{equation}
where $h(\rho)=\frac{\rho^{\gamma-1}-1}{\gamma-1}$ is the enthalpy.
Note that the speed of sound is $c=\rho^\frac{\gamma-1}{2}=\sqrt{(\gamma-1)h(\rho)+1}$.

\smallskip
\subsection{Entropy Solutions and the Rankine-Hugoniot Conditions}
We consider solutions with shocks, which satisfy \eqref{selfSimEuler}--\eqref{selfSimEuler-moment}
in the following weak sense:

\begin{definition}\label{weakSolEulerInOpenDomain}
Let $\Lambda\subset\bR^2$ be a
domain. Then $(\rho,\vv)\in L^{\infty}(\Lambda)$ is an entropy solution of
system \eqref{selfSimEuler}--\eqref{selfSimEuler-moment} if the following conditions hold{\rm:}
\begin{enumerate}
\item[\rm (i)] $(\rho,\vv)$ is a weak solution{\rm :} For any test functions
$\phi\in C^\infty_{\rm c}(\Lambda)$ and $\zeta\in C^\infty_{\rm c}(\Lambda;\,\bR^2)$,
\begin{align}\label{weakSolEuler-Equal-cmass}
 & \int_\Lambda (\rho\vv\cdot\grad \phi -2\rho\phi)\, {\rm d}\xxi=0,  \\
 \label{weakSolEuler-Equal-cmoment}
  & \int_\Lambda \big(\rho\vv\otimes \vv: D\zeta -3\rho\vv\cdot\zeta+p\,\divg \zeta\big)\, {\rm d}\xxi=0,
\end{align}
where we have used the notation $\displaystyle A:B=\sum_{i, j=1}^2a_{ij} b_{ij}$ for $2\times 2$ matrices $A$ and $B$.
\item[\rm (ii)] $(\rho,\vv)$ satisfies the entropy condition{\rm :} For any non-negative test function
$\psi\in C^\infty_{\rm c}(\Lambda)$,
\begin{align}\label{weakSolEuler-Equal-entr}
 & \int_\Lambda \Big(\big(\frac 12\rho|\vv|^2+\rho e(\rho)+p(\rho)\big)\vv\cdot\grad \psi
   -2\big(\rho|\vv|^2+\rho e(\rho)+p(\rho)\big)\psi\Big)\, {\rm d}\xxi\ge 0,
\end{align}
where the internal energy $e(\rho)$ is defined by $p(\rho)=\rho^2e'(\rho)$, i.e., $e(\rho)=\frac{\rho^{\gamma-1}}{\gamma(\gamma-1)}$
for the polytropic case with $p(\rho)=\frac{\rho^\gamma}{\gamma}$.
\end{enumerate}
\end{definition}

Suppose that $S$ is a smooth curve in domain $\Lambda$.
An entropy solution of \eqref{selfSimEuler}--\eqref{selfSimEuler-moment},
which is $C^1$ in the open domains near $S$ and $C^0$ up to $S$ on both sides,
satisfies the following Rankine-Hugoniot (R-H) conditions on $S$:
\begin{equation}\label{RH}
[\rho\vv\cdot\bn]=0,\quad\, [(\rho\vv\cdot\bn)\vv+p\bn]={\bf 0}\qquad\,\,\mbox{on $S$},
\end{equation}
where $\bn$ is a unit normal to $S$, and $[\,\cdot\,]$ denotes the difference of the concerned quantity across $S$.
Denote by $(\rho^\pm, \vv^\pm)$ the values of solution $(\rho, \vv)$ on the $\pm$ sides of $S$.
Assume that $\rho^\pm >0$ on $S$ and some of $(\rho, \vv)$ are discontinuous across $S$.

If $\rho^+\vv^+\cdot\bn=0$ on $S$, then $\rho^-\vv^-\cdot\bn=0$ on $S$ by the first equality
in \eqref{RH}, and $[p]=0$ from the second equality in \eqref{RH}
which implies that $[\rho]=0$. This discontinuity is called a vortex sheet.
In this case, $[\vv\cdot\bt]\ne 0$ on $S$, unless the solution is continuous across $S$,
where $\bt$ is a unit tangent vector to $S$.

If $\rho^+\vv^+\cdot\bn\ne0$ on $S$, then $\rho^-\vv^-\cdot\bn\ne 0$ on $S$ by the first equality in \eqref{RH}.
This discontinuity is called a shock. In this case, $\rho\vv \cdot\bn$ is continuous across $S$ from
the first equality in \eqref{RH},
and $[\vv\cdot\bt]=0$ and $[p]=-\rho\vv\cdot\bn[\vv\cdot \bn]$ from the second equality in \eqref{RH}.
This implies that $[\vv \cdot\bn]\ne 0$, unless the solution is continuous across $S$, so that
$[\rho]\ne 0$ from the first equality in \eqref{RH}.
Also, it follows from $\rho^\pm>0$ on $S$ that $(\vv^+\cdot\bn)(\vv^-\cdot\bn)>0$ on $S$.

Thus, we have shown that the following properties hold, in addition to \eqref{RH},
in the case $\rho^\pm>0$ on $S$:
 \begin{align}
&\mbox{Shock:}\hspace{-8pt} && (\vv^+\cdot\bn)(\vv^-\cdot\bn)>0, \;\; [\vv\cdot\bt]=0,
 \;\; [\vv\cdot\bn]\ne 0, \;\; [\rho]\ne 0; \label{shock-RHp}   \\
&\mbox{Vortex sheet:}\hspace{-8pt} && \vv^+\cdot\bn=\vv^-\cdot\bn=0, \;\; [\vv\cdot\bt]\ne0,
 \;\;  [\rho]= 0. \label{vortexSheet-RHp}
 \end{align}
Furthermore, the entropy condition \eqref{weakSolEuler-Equal-entr} is required across shock $S$
separating the smooth states $(\rho^\pm, \vv^\pm)$ defined in domains $\Lambda^\pm$ with $\rho^\pm>0$.
Then it follows from the direct calculation through \eqref{weakSolEuler-Equal-entr}--\eqref{shock-RHp}
and the choice of orientation of the unit normal $\nnu$ on $S$ to point
from $\Lambda^-$ to $\Lambda^+$ that,  on $S$,
\begin{equation}\label{shock-entropy-RHp}
\mbox{If } \vv^-\cdot\nnu>0, \;\; \mbox{ then }
\;\vv^-\cdot\nnu>\vv^+\cdot\nnu>0,\,\,\rho^-<\rho^+, \;\;
\vv^-\cdot\nnu> c^-,\;\; \vv^+\cdot\nnu< c^+,
\end{equation}
where the last two inequalities are shown {\it e.g.}, in \cite[Theorem 2.2]{Serre-2007}
(the argument is given there for steady solutions;
the proof applies to the self-similar case because the self-similar Rankine-Hugoniot conditions are the same),
since $p(\rho)$ is increasing and convex, implied by $p(\rho)=\frac{\rho^\gamma}{\gamma}$ with $\gamma>1$.

We also give a definition of entropy solutions of system \eqref{selfSimEuler}--\eqref{selfSimEuler-moment} in $\Lambda$
with the slip boundary conditions:
\begin{equation}\label{Euler-slipBC-def}
 \vv\cdot\nnu=0\qquad\mbox{on $\partial\Lambda$},
\end{equation}
where $\nnu$ is the outer normal to $\partial\Lambda$.
The motivation is the following:
Suppose that $(\rho,\vv)\in C^1$ and $\partial\Lambda\in {\rm Lip}$ satisfy \eqref{selfSimEuler}--\eqref{selfSimEuler-moment}
in $\Lambda$ and \eqref{Euler-slipBC-def}.
Then it follows that \eqref{weakSolEuler-Equal-cmass} and
\begin{equation}\label{weakSolEuler-Equal-cmoment-BC}
   \int_\Lambda \big(\rho\vv\otimes \vv: D\zeta -3\rho\vv\cdot\zeta+p\,\divg \zeta\big)\,{\rm d}\xxi
   -\int_{\partial \Lambda} p\,\zeta\cdot\nnu\,{\rm  d}l=0
\end{equation}
are satisfied for any test function
$\zeta\in C^\infty_{\rm c}(\bR^2;\,\bR^2)$.
Note that here the test function is not required to vanish on $\partial\Lambda$.

From now on, we define the $r$-neighborhood of a curve $\Gamma$ as
\begin{equation}\label{r-nbhd}
\mathcal{N}_r(\Gamma):=\{\xxi\, : \,\dist\{\xxi,\Gamma\}<r\}.
\end{equation}
Based on that, we define the notion of entropy solutions of the boundary value problem
\eqref{selfSimEuler}--\eqref{selfSimEuler-moment} and \eqref{Euler-slipBC-def}
as follows:

\begin{definition}\label{weakSolEulerSlipBC}
Let $\Lambda\subset\bR^2$ be a domain with a Lipschitz boundary $\partial\Lambda$.
Let $(\rho,\vv)\in L^{\infty}(\Lambda)$, and let $\rho\in BV_{\rm loc}(\Lambda\cap {\mathcal N}_r(\partial\Lambda))$ 
for some $r>0$.
Then $(\rho,\vv)$ is an entropy solution of system \eqref{selfSimEuler}--\eqref{selfSimEuler-moment}
with slip boundary condition \eqref{Euler-slipBC-def} if  \eqref{weakSolEuler-Equal-cmass},  \eqref{weakSolEuler-Equal-cmoment-BC}, and
\eqref{weakSolEuler-Equal-entr}
are satisfied for any test functions
$\phi\in C^\infty_{\rm c}(\bR^2)$, $\zeta\in C^\infty_{\rm c}(\bR^2;\,\bR^2)$, and $\psi\in C^\infty_{\rm c}(\Lambda)$ with $\psi\ge 0$, respectively.
\end{definition}

\medskip
\section{Low Regularity of Self-Similar Solutions of the Riemann
Problems with Shocks for the Isentropic Euler System}
\label{lowRegSect}

In this section, we first formulate a general framework for analyzing
the low regularity of entropy solutions ({\it i.e.},
self-similar solutions of admissible structure)
of the Riemann problems
in Definition \ref{RiemannProblSolutStruct-1},
motivated by the solutions of the physically fundamental Riemann problems described
in \S \ref{SectionApplications} below.
Then we establish our main theorem, Theorem \ref{lower-reg-Th},
for the entropy solutions $(\rho, \vv)$ in the general framework,
which will be applied to
understanding the low regularity of the solutions of the Riemann
problems including the regular shock reflection problem, the Prandtl reflection problem,
the Lighthill diffraction problem, and the Riemann problem with four shock interactions.
These are achieved by carefully analyzing the vorticity function $\omega:=\partial_1v_2-\partial_2 v_1$
for the pseudo-velocity $\vv=(v_1, v_2)$.

\begin{figure}[!h]
  \centering
\centering
     \includegraphics[width=0.4\textwidth]{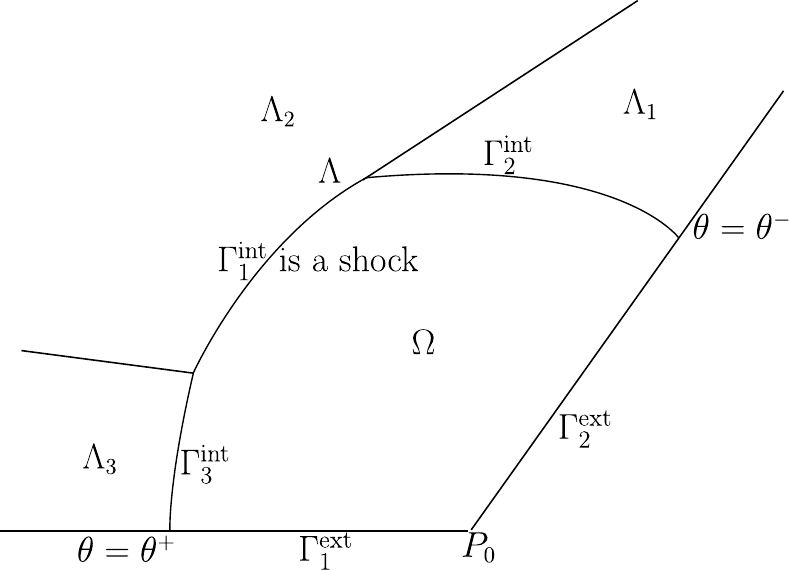}
  \caption{The Riemann Problem in a General Setting}
  \label{figure:general setting}
\end{figure}

More precisely, as shown in Fig. \ref{figure:general setting}, we consider a Riemann problem in a self-similar wedge domain $\Lambda$ in $\bR^2$
with the wedge-vertex being the origin, which also includes the case of the whole space. That is,
in polar coordinates,
\begin{equation}\label{domainLambda}
 \mbox{either $\,\,\Lambda=\bR^2\,\,\,$ or $\,\,\Lambda=\big\{(r, \theta)\;:\; \theta^-<\theta<\theta^+\big\}$}.
\end{equation}
For a Riemann problem in a domain $\Lambda$ whose boundary contains a wedge-boundary,
the initial data are the given constant velocity and density in each sub-sector of $\Lambda$
such that the Rankine-Hugoniot conditions hold on the lines separating them.
Motivated by the expected configurations of self-similar solutions of the Riemann problems in \S \ref{SectionApplications},
we consider solutions in self-similar coordinates of
the following form, in terms of the density and pseudo-velocity $(\rho, \vv)$:

\begin{definition}\label{RiemannProblSolutStruct-1}
We say that the vector function $(\rho, \vv)$ on $\Lambda$ is an entropy solution of the Riemann problem
of admissible structure if
$(\rho,\vv)\in L^{\infty}(\Lambda)$, with
$\rho\in BV_{\rm loc}(\Lambda\cap {\mathcal N}_r(\partial\Lambda))$ for some $r>0$,
is an entropy solution of system \eqref{selfSimEuler}--\eqref{selfSimEuler-moment}
with the slip boundary condition \eqref{Euler-slipBC-def}
in the sense of {\rm Definition \ref{weakSolEulerSlipBC}}, which satisfies
the following properties{\rm :}
\begin{enumerate}[\rm (i)]
  \item \label{RiemannProblSolutStruct-1-i1}
  There exist an open, bounded, connected set
    $\Omega\subset\Lambda$, an integer $M\ge 1$, and open connected sets
    $\Lambda_i$, $i=1,\cdots, M$, pairwise disjoint, such that
    $$
    \overline{\Lambda}\setminus\Omega=\cup_{i=1}^M\overline{\Lambda_i},
    $$
and $(\rho, \vv)$ is a constant state in each $\Lambda_i$,
i.e., $(\rho, \vv)(\xxi)=(\rho_i, \uu_i-\xxi)$ in $\Lambda_i$, where $\rho_i>0$ is a constant and
    $\uu_i$ is a constant vector.

\item \label{RiemannProblSolutStruct-1-i2}
If, for $i\ne j$, sets $\Lambda_i$ and $\Lambda_j$ have common boundary within $\Lambda${\rm ,} {\it i.e.}, 
if $\partial\Lambda_i\cap \partial\Lambda_j\cap\Lambda\ne\emptyset$,
then the corresponding constant states are not equal to each other{\rm :} $(\rho_i, \vv_i)\ne (\rho_j, \vv_j)$.

\item \label{RiemannProblSolutStruct-1-i3}
$\partial\Omega$ is Lipschitz.
Denote
$$
    \GammaExt=\partial\Omega\cap\partial\Lambda, \qquad
    \GammaInt=\partial\Omega\cap\Lambda,
$$
so that $\partial\Omega=\GammaExt\cup\GammaInt$. It is possible that $\GammaExt=\emptyset$
{\rm (}in particular, this is true when $\partial\Lambda=\emptyset${\rm )}.

\item \label{RiemannProblSolutStruct-1-i4}
If $\partial\Lambda\ne\emptyset$, then ${\GammaExt}=\cup_{i=1}^\Nex \overline{\GammaExt_{i}}$,
where $\Nex\ge 1$, each $\GammaExt_{i}$ is
       a relatively open segment of straight line.
Segments  $\{\GammaExt_{i}\}_{i=1}^\Nex$
are disjoint and, if $\GammaExt_{i}$  and $\GammaExt_{j}$ have a common endpoint,
       then the interior angle {\rm (}for $\Omega${\rm )} is within $(0, 2\pi)$ at that point.
       Since $\Lambda$ is a wedge domain
       with the wedge-vertex at the origin, only
       one pair among segments $\{\GammaExt_{i}\}$ may have a common endpoint that is
       $\PZer=(0,0)$.
\item \label{RiemannProblSolutStruct-1-i5}
    $\overline{\GammaInt}=\cup_{i=1}^\Nint \overline{\GammaInt_{i}}$, where  $\Nint\ge 1$, 
    each $\GammaInt_{i}$ is a relatively open segment of curve,
    and segments  $\{\GammaInt_{i}\}_{i=1}^\Nint$ are disjoint.
    Each $\GammaInt_{i}$
        is $C^2$  in its relative interior and $C^1$ up to the endpoints. Moreover,
        $\PZer\notin\overline{\GammaInt}$.
\item \label{RiemannProblSolutStruct-1-i6-0}
$\vv\in C\big(\Nbhd_\sigma(\GammaInt)\cap\overline\Omega\big)$ for some $\sigma>0$.

\item\label{RiemannProblSolutStruct-1-i6}
$\GammaInt_{1}$ is a shock and $\vv\cdot\nnu\le -C^{-1}$ on $\GammaInt_{1}$,  and
$\vv\cdot\nnu\le 0$
on $\GammaInt_{i}$ for $i=2,\cdots, \Nint$, where $\vv$  is taken from the $\Omega$--side
 and $\nnu$ denotes the outer unit normal with respect to $\Omega$.
       \item
\label{RiemannProblSolutStruct-1-i7}
There exists a point $\hat P$ in the relative interior of $\GammaInt_{1}$ such that the curvature of $\GammaInt_{1}$
is non-zero at $\hat P$, and $(\vv\cdot\ttau)(\hat P)\ne 0$, where $\ttau(\hat P)$ is a unit
tangent  vector to
$\GammaInt_{1}$ at $\hat P$.
  \end{enumerate}
\end{definition}

\smallskip
Then we have the following main theorem.

\smallskip
\begin{theorem}[Main Theorem]\label{lower-reg-Th}
Let $(\rho, \vv)$ be a solution of the Riemann problem
of admissible structure in the sense of {\rm Definition \ref{RiemannProblSolutStruct-1}}.
Assume that $(\rho, \vv)$  satisfy
\begin{enumerate}[\rm (i)]
\item
\label{lower-reg-Th-i2}
$(\rho, \vv)\in C^1((\Nbhd_\sigma(\GammaInt)\cap\overline\Omega)\setminus\partial_{\rm p}\GammaInt)\cap C^{0,1}(\Nbhd_\sigma(\GammaInt)\cap\overline\Omega)$
for some $\sigma>0$, where $\partial_{\rm p}\GammaInt$
denotes the set of endpoints of the curve segments $\GammaInt_i$ for $i=1, \cdots, \Nint$.
\item
\label{lower-reg-Th-i3}
$|\vv|\le C_0 $ and $C_0^{-1}\le \rho\le C_0$ in $\Omega$ for some $C_0\ge 1$.
\item
\label{lower-reg-Th-i4}
The flow is subsonic on $\GammaInt_{1}$ from the $\Omega$--side{\rm :}
$|\vv|<c$ on $\GammaInt_{1}$, where $c=\rho^\frac{\gamma-1}{2}$ is the speed of sound.
\end{enumerate}
Then   $\vv\notin H^1(\Omega)$.
\end{theorem}

%%%%%%%%%%%%%%%%%%%%%%%%%%%%%%%%%%%%%%%%%%%%%%%%%%
%%%%%%%%%%%%%%%%%%%%%%%%%%%%%%%%%%%%%%%%%%%%%%%%%%
\begin{proof} 
We divide the proof into seven steps. In the first six steps, we prove that 
both $\rho$ and $\vv$
can not be in $H^1(\Omega)$:
\begin{equation}\label{v-rho-l2-b}
(\rho, \vv)\notin H^1(\Omega).
\end{equation}
This will be proved by showing that, under the assumption that 
\begin{equation}\label{v-rho-l2}
(\rho, \vv)\in H^1(\Omega),
\end{equation}
the calculations of vortical singularity can rigorously be justified, which leads to the contradiction.
Then, in Step 7, we prove that $\vv$ is not in $H^1(\Omega)$ indeed: $\vv\notin H^1(\Omega)$.

\medskip
{\bf 1.} 
In Steps 1--6 of this proof, we assume that \eqref{v-rho-l2} holds.
Then vorticity $\omega=\nabla\times \vv$ satisfies $\omega\in L^2(\Omega)$
and
\begin{equation}\label{omega-over-rho-l2}
X:=\frac\omega\rho\in L^2(\Omega),
\end{equation}
where we have used \eqref{v-rho-l2} and assumption \eqref{lower-reg-Th-i3}.
In addition, using \eqref{v-rho-l2} and assumptions (i)--(ii), we see that
\begin{equation}\label{conMass-a-e}
\mbox{\it The left-hand sides of equations \eqref{selfSimEuler} and \eqref{Cons-Moment-nonDiv} are in $L^2(\Omega)$ and the equations hold a.e. in $\Omega$.}
\end{equation}

\smallskip
{\bf 2.} We first formally derive the equations and identities that vorticity $\omega$ satisfies, and then prove them rigorously.

Taking the $\curl$ of \eqref{Cons-Moment-nonDiv}, we formally obtain the equation:
\begin{equation}\label{vorticityEq-1}
 \vv\cdot\grad \omega + (1+\divg\, \vv)\,\omega=0.
\end{equation}
Combining with the first equation of \eqref{selfSimEuler}, we formally have
\begin{equation}\label{vorticityRho-eq0}
\vv\cdot\grad \big(\frac{\omega}{\rho}\big)
=\frac{\omega}{\rho}.
\end{equation}

If $f\in C^1(\bR)$, multiplying the last equation by $f'(\frac{\omega}{\rho})$ and then combining
with the first equation of \eqref{selfSimEuler}, we formally derive
\begin{equation}\label{vorticityRho-Renorm}
\divg\big(\rho f(\frac\omega\rho) \vv\big)
= \rho g(\frac\omega\rho),
\end{equation}
where
\begin{equation}\label{f-g-vorticityRho-Renorm}
 g(s):=sf'(s)-2f(s).
\end{equation}

\smallskip
Now we are going to show that equation \eqref{vorticityRho-Renorm}
holds under the present assumptions in the weak sense defined as follows:
Using notation \eqref{omega-over-rho-l2}, equation \eqref{vorticityRho-Renorm} can be written as
\begin{equation}\label{vorticityRho-Renorm-X}
\divg\left(\rho f(X) \vv\right)= \rho g(X)
\end{equation}
for any $f\in C^1(\bR)$ with $g$ defined by \eqref{f-g-vorticityRho-Renorm}.
From the weak form \eqref{weakSolEuler-Equal-cmass} of the conservation law of mass,
which holds for all $\phi$ specified in Definition \ref{weakSolEulerSlipBC},
and
the regularity assumption \eqref{v-rho-l2}, we obtain that
$\rho\vv\cdot\nnu=0$ holds ${\mathcal H}^1$--{\it a.e.} on  $\GammaExt$
in the sense of traces.
Then, using assumption \eqref{lower-reg-Th-i3},
we recover condition \eqref{Euler-slipBC-def} on
$\partial \Lambda\cap\partial\Omega$:
\begin{equation}\label{bcOnWedgeSym}
 \vv\cdot\nnu=0\qquad\,\,{\mathcal H}^1\mbox{--a.e. on }\;\GammaExt.
\end{equation}

We next show that, under the present assumptions, \eqref{vorticityRho-Renorm-X} holds weakly in $\Omega$
in the sense that
\begin{equation}\label{vorticityEq-weak}
 \int_\Omega \big( \rho f(X)\vv\cdot\grad\zeta +\rho g(X)\zeta\big) {\rm d}\xxi -
 \int_{\GammaInt}\rho f(X)  \vv\cdot\nnu\, \zeta\,{\rm d}l=0\qquad\mbox{ for all }\; \zeta\in C^\infty(\bR^2)
\end{equation}
for any $f\in C^1(\bR)$ with $\|f'\|_{C^{0,1}(\bR)}<\infty$.

Notice that
the boundary integral in \eqref{vorticityEq-weak} is taken along a part of $\partial\Omega$.
Formally, the boundary integral along the remaining part $\GammaExt$ of $\partial\Omega$
is expected to vanish by \eqref{bcOnWedgeSym}. Equation \eqref{vorticityEq-weak} shows that
$$
\rho f(X) \vv\cdot\nnu=0
\qquad \mbox{on $\GammaExt$ in the weak sense.}
$$
Furthermore, the integrand in the boundary integral in \eqref{vorticityEq-weak} is well-defined by
       assumption \eqref{lower-reg-Th-i2}  and {\rm Definition \ref{RiemannProblSolutStruct-1}\eqref{RiemannProblSolutStruct-1-i5}}.

\medskip
{\bf 3.} To prove \eqref{vorticityEq-weak}, we introduce smooth approximations of $(\rho, \vv)$.
Since we work in the bounded domain $\Omega$ and  need the boundary condition $\vv\cdot \nnu=0$ on $\GammaExt$ and other properties
to hold for the approximating functions for the argument below, we first construct a specific sequence
of smooth approximations of $(\rho, \vv)$.

We first extend $\vv$ by the reflection across the straight boundary segments $\GammaExt_i$, $i=1,\cdots, \Nex$,
so that the normal component is extended by the odd reflection and
condition \eqref{bcOnWedgeSym} is used to conclude that the extended function remains in $H^1$,
and the tangential component is extended by the even reflection.

Denote by $\NuExtI$ the unit inner normal vector to $\GammaExt_i$ with respect to $\Lambda$, $i=1,\cdots, \Nex$, respectively,
and by $\TauExtI$ the unit tangential vectors to $\GammaExt_i$.

\newcommand{\hatV}{W}
For each $r>0$ and $i=1, \cdots, \Nex$, we first define an extension of $\vv$ from $\Nbhd_{r}(\GammaExt_i)\cap\Omega$ across $\GammaExt_i$ by reflection.
However,
there is the following issue: Denote
$$
\hatV_i^{r}:=\big\{\xxi-s\NuExtI\;:\; \xxi\in\GammaExt_i, \;s\in(-r, 0),\; \xxi+s\NuExtI\in\Omega\big\},
$$
the image of $\Nbhd_{r}(\GammaExt_i)\cap\Omega$ under the reflection across $\GammaExt_i$.
We note that, if $\PZer=(0,0)$ is a common endpoint of $\GammaExt_{i}$ and $\GammaExt_{j}$,
and the interior angle $\theta_{\PZer}$ of $\Lambda$ at $\PZer$ is larger than $\pi$, {\it i.e.},
$\theta_{\PZer}\in (\pi, 2\pi)$, then $\hatV_i^{r}\cap \hatV_j^{r}\ne\emptyset$  for any $r>0$.
Moreover, if $\theta_{\PZer}\in (\frac 32\pi, 2\pi)$, then $\hatV_i^{r}\cap \Omega\ne\emptyset$ for any $r>0$.
Of course, in general, the extension of $\vv$ by reflection across $\GammaExt_{i}$ into $\hatV_i^{r}$
does not match with $\vv$ in $\hatV_i^{r}\cap \Omega$ or with
extension of $\vv$ by reflection across $\GammaExt_{j}$ in $\hatV_i^{r}\cap \hatV_j^{r}$,
so the region needs to be restricted.

In the case above, for each $r>0$, we have
\begin{equation}\label{non-intersect}
  (\hatV_i^{r}\cap \Omega)\setminus B_{Lr}(\PZer)=\emptyset, \quad
  (\hatV_i^{r}\cap \hatV_j^{r})\setminus B_{Lr}(\PZer)=\emptyset \qquad\mbox{for some }
  L=L(\theta_{\PZer})\ge 0.
\end{equation}
In fact, this is true with $L=\mbox{cosec}\, {\theta_{\PZer}}$
if $\theta_{\PZer}\in (\pi, 2\pi)$ and $L=0$ if $\theta_{\PZer}\in (0, \pi]$.
Also, there exists $r_1>0$ such that, for $i=1, \cdots, \Nex$,
\begin{equation}\label{def-R1}
\GammaExt_{i}\cap B_{(L+1)r_1}(\PZer) = \emptyset \qquad\mbox{if $\PZer$ is not an endpoint of  }
\GammaExt_{i}.
\end{equation}
Now, for each $r\in (0, r_1]$, define
$$
V_i^{r}:= \hatV_i^{r} \setminus  \overline{B_{Lr}(\PZer)},
$$
which is an open connected set.
We extend $\vv$ to $V_i^r$ by
\begin{equation}\label{defExtReflSymm}
\vv(\xxi-s\NuExtI):=-(\vv\cdot\NuExtI) (\xxi+s\NuExtI) \NuExtI
   +(\vv\cdot\TauExtI) (\xxi+s\NuExtI)\TauExtI
\end{equation}
for all $\xxi\in \GammaExt_i$ and $s\in (0, r)$
such that $\xxi+s\NuExtI\in\Omega\setminus  \overline{B_{Lr}(\PZer)}$.
From the definitions of $\hatV_i^r$ and $V_i^{r}$, since segment $\GammaExt_i$ lies on the line passing through $\PZer$,
the expression above defines $\vv$ on the whole region $V_i^r$.
Also, it follows that, for all $\xxi\in \GammaExt_i$ and $s\in (0, r)$ such that $\xxi+s\NuExtI\in\Omega\setminus  \overline{B_{Lr}(\PZer)}$,
\begin{equation}\label{defExtReflSymm-11}
\begin{split}
&(\vv\cdot\NuExtI) (\xxi-s\NuExtI):=-(\vv\cdot\NuExtI) (\xxi+s\NuExtI),\\[1mm]
&(\vv\cdot\TauExtI)(\xxi-s\NuExtI):=(\vv\cdot\TauExtI) (\xxi+s\NuExtI),
\end{split}
\end{equation}
that is, $\vv\cdot\NuExtI$ is extended by the odd reflection, and $\vv\cdot\TauExtI$
 is extended by the even reflection across $\GammaExt_i$.

Similarly, for each $i=1, \cdots, \Nex$ and $r\in (0, r_1]$, we extend $\rho$ to $V_i^r$ by the even reflection
across $\GammaExt_i$:  For all $\xxi\in \GammaExt_i$ and $s\in (0, r)$ such that $\xxi+s\NuExtI\in\Omega\setminus  \overline{B_{Lr}(\PZer)}$,
\begin{equation}\label{defExtRho}
\rho (\xxi-s\nnu):=\rho (\xxi+s\nnu),
\end{equation}
where we recall that $\GammaExt_i$ is a relatively open line segment.

Using \eqref{v-rho-l2}
and \eqref{lower-reg-Th-i3},
and noting that $\vv\cdot\nnu=0$ to make the odd extension of $\vv\cdot\nnu$ across $\GammaExt_i$,
we find that the extended $(\rho,\vv)$ satisfies $(\rho, \vv)\in H^1((\Omega\cup \overline{V_i^{r_1}})^0)\cap L^\infty(\Omega\cup \overline{V_i^{r_1}})$
and $\rho\ge C_0^{-1}$ on $\Omega\cup \overline{V_i^{r_1}}$.

Combining the extensions for $i=1,\cdots, \Nex$,  and using \eqref{non-intersect}--\eqref{def-R1},
we obtain that, for each $r\in(0, r_1]$, $(\rho, \vv)$ is extended to the domain:
  \begin{equation}\label{defExtDomain}
 \OmegaPM_{r}=\big(\Omega\cup
 (\cup_{i=1}^\Nex\overline{V_i^{r}} )\big)^0,
\end{equation}
and the extension satisfies $(\rho,\vv)\in H^1(\OmegaPM_{r})\cap L^\infty(\OmegaPM_{r})$  and
$\rho\ge C_0^{-1}$ on $\OmegaPM_{r}$.

Moreover, by \eqref{conMass-a-e} and the explicit structure
\eqref{defExtReflSymm}--\eqref{defExtRho}
of the extension, it follows that property \eqref{conMass-a-e}
holds for the extended $(\rho, \vv)$ in $\OmegaPM_{r}$ for any $r\in(0, r_1]$:
\begin{equation}\label{conMass-a-e-Ext}
\begin{aligned}
&\mbox{\it The left-hand sides of Eqs. \eqref{selfSimEuler} and \eqref{Cons-Moment-nonDiv} are }\\
&\mbox{\it in $L^2(\OmegaPM_{r})$, and  the equations hold a.e. in $\OmegaPM_{r}$.}
\end{aligned}
\end{equation}
Indeed, if $\xxi\in \GammaExt$ and $s\in (0, r)$, and if $(\rho, \vv)$ is differentiable at $\xxi+s\nnu$,
then clearly $(\rho, \vv)$ is differentiable at $\xxi-s\nnu$, and \eqref{selfSimEuler} and \eqref{Cons-Moment-nonDiv} hold
at $\xxi-s\nnu$, which can be seen by the explicit calculation or using the standard symmetries of
the isentropic Euler system \eqref{selfSimEuler}--\eqref{selfSimEuler-moment}.

Thus, we have shown the following results in this step:
\begin{lemma}\label{approxVelocityLemma}
Under assumptions \eqref{lower-reg-Th-i2}--\eqref{lower-reg-Th-i3}
of {\rm Theorem \ref{lower-reg-Th}} and \eqref{v-rho-l2},
there exists $r_1>0$ depending only on $\Omega$ such that, for any $r\in(0, r_1]$,
there is an extension of $(\rho,\vv)$ into $\OmegaPM_{r}$, still denoted as $(\rho,\vv)$, so that
\begin{enumerate}[\rm (a)]
\item \label{approxVelocityLemma-i1}
$(\rho, \vv)\in H^1(\OmegaPM_{r})\cap L^\infty(\OmegaPM_{r})$ with $\rho\ge C_0^{-1}$ in $\OmegaPM_{r};$

\item \label{approxVelocityLemma-i2}
For any $\xxi\in \GammaExt$ and $s\in (0, r)$ such that $\xxi+s\nnu\in\Omega\setminus \overline{B_{Lr}(\PZer)}$ for $L$
from \eqref{non-intersect}, the following odd/even reflection properties hold{\rm :}
\begin{equation}\label{defExtReflAll}
\begin{cases}
(\vv\cdot\nnu) (\xxi-s\nnu)=-(\vv\cdot\nnu) (\xxi+s\nnu),\\[1mm]
   (\vv\cdot\ttau)(\xxi-s\nnu)=(\vv\cdot\ttau) (\xxi+s\nnu), \\[1mm]
\rho (\xxi-s\nnu)=\rho (\xxi+s\nnu).
\end{cases}
\end{equation}

\item \label{approxVelocityLemma-i3} Property \eqref{conMass-a-e-Ext}
holds.
\end{enumerate}
\end{lemma}

{\bf 4.} We now show that equation \eqref{vorticityRho-Renorm-X} is satisfied in the weak sense:

\begin{lemma}\label{vorticityEqnProp}
Under the assumptions  of {\rm Theorem \ref{lower-reg-Th}} and  \eqref{v-rho-l2},
equation \eqref{vorticityRho-Renorm-X} holds weakly in the sense
\eqref{vorticityEq-weak} for each $f\in C^1(\bR)$  with $\|f'\|_{C^{0,1}(\bR)}<\infty$
and $g$ defined by \eqref{f-g-vorticityRho-Renorm}.
\end{lemma}

This can be proved as follows:
We reduce $r_1$ to obtain, in addition to \eqref{def-R1}, that, for  $j=1,\cdots, \Nint$,
\begin{equation}\label{def-R1-1}
\GammaInt_{j}\cap B_{(L+1)r_1}(\PZer) = \emptyset\qquad \mbox{and} \qquad r_1<\sigma,
\end{equation}
where we have used that  $\PZer\notin\overline{\GammaInt}$ by
Definition \ref{RiemannProblSolutStruct-1}\eqref{RiemannProblSolutStruct-1-i5},
$\sigma$ is from assumption \eqref{lower-reg-Th-i2}
of Theorem \ref{lower-reg-Th},
and $L$ is from \eqref{non-intersect}.
Fix $f\in C^1(\bR)$  with $\|f'\|_{C^{0,1}(\bR)}\le C$,
and let $g$ be defined by \eqref{f-g-vorticityRho-Renorm}.

\medskip
{\bf 4.1.} We first prove \eqref{vorticityEq-weak} in the case
when the smooth function $\zeta$ satisfies that there exists $r\in (0, r_1)$ such that
\begin{equation}\label{case-1-vortEq}
 \zeta\equiv 0\qquad \mbox{in $\Nbhd_r(\GammaInt)\cup B_{2Lr}(\PZer)$}.
\end{equation}
Thus, we need to consider $(\rho, \vv)$ only in the region:
\begin{equation}\label{Omega-r-def}
\Omega^r:=\Omega\setminus \big(\overline{\Nbhd_r(\GammaInt)}\cup \overline{B_{2Lr}(\PZer)}\big).
\end{equation}

Note that $\Omega^r\Subset \OmegaPM_r$. Let
$$
\delta:=\dist(\partial \OmegaPM_r,\; \Omega^r).
$$
Then $\delta\in (0, r]$.
We now mollify $(\vv, \omega)$ and $\rho$ in $\Omega^r$, by using the extension of $(\rho,\vv)$ into
$\OmegaPM_r$ constructed in Lemma \ref{approxVelocityLemma} and the corresponding extension
of $\omega=\partial_1v_2-\partial_2 v_1$.
In order to achieve that the mollified $\vv$ satisfies the boundary condition \eqref{bcOnWedgeSym},
we  use $\eta\in C^\infty_{\rm c}(\bR^2)$ with  $\int_{\bR^2}\eta(\xxi)\,{\rm d}\xxi=1$ of form $\eta(\xxi)=g(|\xxi|)$
for some $g\in C^\infty_{\rm c}(\bR)$, for example, the standard mollifier.
Then we define $\displaystyle \eta_\eps(\xxi)=\frac 1{\eps^2}\eta(\frac \xxi\eps)$ for $\eps>0$, and
denote $F_\eps:=F*\eta_\eps$ for various functions $F$, specifically
$$
\vv_\eps=\vv*\eta_\eps,\quad \omega_\eps=\omega*\eta_\eps,\quad \rho_\eps=\rho*\eta_\eps
\quad (h(\rho))_\eps=h(\rho)*\eta_\eps\qquad\,\,\mbox{for $\eps\in (0, \frac \delta 2)$}.
$$
Now, since  $(\rho, \vv)\in H^1(\OmegaPM_{r})\cap L^\infty(\OmegaPM_{r}) $
with $\rho\ge C^{-1}$ by  Lemma \ref{approxVelocityLemma}(\ref{approxVelocityLemma-i1})--(\ref{approxVelocityLemma-i2}),
and $\Omega^r\Subset\OmegaPM_r$, we obtain that, for each $\eps\in(0,\frac{\delta}{2})$,
\begin{equation}\label{conw-v-rho-omega}
\begin{aligned}
&\rho_\eps, \, \vv_\eps, \,(h(\rho))_\eps\in H^1(\Omega^r)\cap L^\infty(\Omega^r),
  \quad \|(\rho_\eps,\,\vv_\eps)\|_{L^\infty(\Omega^r)}\le \|(\rho,\, \vv)\|_{L^\infty(\OmegaPM_{r})},\\[1mm]
&(\rho_\eps,\,\vv_\eps) \to (\rho,\,\vv) \,\,\,\,\mbox{in $H^1(\Omega^r)$}, \quad\,
   \omega_\eps \to \omega \,\,\,\,\mbox{in $L^2(\Omega^r)\qquad\,\,$  as $\eps\to 0$}.
\end{aligned}
\end{equation}
Moreover, from \eqref{defExtReflSymm}--\eqref{defExtReflSymm-11} and $\eta(\xxi)=g(|\xxi|)$, we have
\begin{equation}\label{bcOnWedgeSymm-minus-Br}
 \vv_\eps\cdot\nnu=0\qquad\mbox{on $\cup_{i=1}^\Nex\GammaExt_i\cap\partial \Omega^r\quad$ for all $\eps\in (0, \delta)$}.
\end{equation}
Now, from
\eqref{conMass-a-e-Ext}, we see that, in $\Omega^r$,
\begin{align}\label{EulerSmoothed-cons-mass}
&\divg(\rho_\eps \vv_\eps)+2\rho_\eps+r^{(1)}_\eps=0, \\
\label{EulerSmoothed}
&(\vv_\eps\cdot \grad) \vv_\eps +\vv_\eps+\grad (h(\rho))_\eps+\rr^{(2)}_\eps=0,
\end{align}
where
\begin{align*}
 r^{(1)}_\eps  = (\divg\,(\rho \vv))_\eps-\divg\,(\rho_\eps \vv_\eps), \qquad
  \rr^{(2)}_\eps  =\big((\vv\cdot \grad) \vv\big)_\eps-(\vv_\eps\cdot \grad) \vv_\eps.
\end{align*}
The functions on the left-hand side of \eqref{EulerSmoothed-cons-mass}--\eqref{EulerSmoothed} are smooth.
Taking the $\curl$ of
\eqref{EulerSmoothed}, we obtain
\begin{equation}\label{vorticity-eqn-eps-with-div}
 \vv_\eps\cdot\grad \omega_\eps + (1+\divg \,\vv_\eps)\omega_\eps +\curl\,\rr^{(2)}_\eps =0
 \qquad\,\,\mbox{in $\Omega^r$}.
\end{equation}
Denote
$$
X^{(\eps)}:=\frac{\omega_\eps}{\rho_\eps}.
$$
Using \eqref{conw-v-rho-omega} and the lower bound of $\rho$ in $\OmegaPM_{r}$ by
 Lemma \ref{approxVelocityLemma}\eqref{approxVelocityLemma-i1},  we obtain that, for $\eps\in(0,\frac{\delta}{2})$,
\begin{equation}\label{conw-X-eps}
  X^{(\eps)}\in L^2(\Omega^r), \qquad\;
  X^{(\eps)}\to X
   \,\,\,\,\mbox{in $L^2(\Omega^r)\,\,\,$ as $\eps\to 0$}.
\end{equation}
We apply the definition of $X^{(\eps)}$ in the first equality, along with
equations \eqref{EulerSmoothed-cons-mass} and
\eqref{vorticity-eqn-eps-with-div} in the second equality, to compute:
\begin{align*}
  \rho_\eps\vv_\eps\cdot \grad X^{(\eps)} &=
 \vv_\eps\cdot\grad \omega_\eps-\nabla\rho_\eps\cdot\vv_\eps X^{(\eps)}  \\
  & =\rho_\eps X^{(\eps)}
  +r^{(1)}_\eps-\curl\,\rr^{(2)}_\eps.
\end{align*}
From this and \eqref{EulerSmoothed-cons-mass}, we have
\begin{align}
\divg\big(\rho_\eps f(X^{(\eps)})\vv_\eps\big)
& =f'(X^{(\eps)})\rho_\eps  \vv_\eps
 \cdot \grad X^{(\eps)} +f(X^{(\eps)}) \divg(\rho_\eps\vv_\eps) \nonumber\\
&=\rho_\eps \big(X^{(\eps)}f'(X^{(\eps)})-2f(X^{(\eps)})\big)+\Rem_\eps,\label{calc-X-eqn-eps}
\end{align}
where
\begin{equation}\label{remainder-full}
  \Rem_\eps=f'(X^{(\eps)})\big(r^{(1)}_\eps-\curl\,\rr^{(2)}_\eps\big)
  - f(X^{(\eps)})r^{(1)}_\eps.
\end{equation}
Recalling definition \eqref{f-g-vorticityRho-Renorm} of $g(\cdot)$, we rewrite
\eqref{calc-X-eqn-eps} as
\begin{align*}
  \divg\big(\rho_\eps f(X^{(\eps)})\vv_\eps\big)& =\rho_\eps g(X^{(\eps)})
  +\Rem_\eps\qquad\mbox{in $\Omega^r$}.
\end{align*}

Let $\zeta\in C^\infty(\bR^2)$ satisfy \eqref{case-1-vortEq}.
Multiply the last equation by $\zeta$, integrate over $\Omega$, and integrate by parts with the use of 
\eqref{case-1-vortEq} and
\eqref{bcOnWedgeSymm-minus-Br} to obtain
\begin{equation}\label{vorticityEq-weak-remainder}
 \int_{\Omega^r} \big(\rho_\eps f(X^{(\eps)}) \vv_\eps\cdot\grad\zeta
 +\big(\rho_\eps g(X^{(\eps)})
  +\Rem_\eps\big)\zeta\big)\, {\rm d}\xi
 =0,
\end{equation}
where we have used \eqref{case-1-vortEq}--\eqref{Omega-r-def} to restrict the domain to $\Omega^r$.

\smallskip
{\bf 4.2.} To send $\eps\to 0$ in \eqref{vorticityEq-weak-remainder}, we note the following facts:
Since $f\in  C^1(\bR)$ with $\|f'\|_{C^{0,1}(\bR)}<\infty$, then,  by \eqref{f-g-vorticityRho-Renorm},
$g\in C(\bR)$ with $\mbox{Lip}(g)<\infty$ on $\bR$.
It follows from \eqref{omega-over-rho-l2} and \eqref{conw-X-eps} that, for all $\eps\in (0, \frac{\delta}{2})$,
\begin{equation}\label{est-fX-gX-eps}
\|f(X^{(\eps)})\|_{L^2(\Omega^r)}+\|f'(X^{(\eps)})\|_{L^\infty(\Omega^r)}
+\|g(X^{(\eps)})\|_{L^2(\Omega^r)}\le C,
\end{equation}
and
\begin{equation}\label{conv-fX-gX-eps}
  (f(X^{(\eps)}),\, g(X^{(\eps)}))  \to\, (f(X), g(X))
   \qquad\mbox{in $L^2(\Omega^r)\,\,$ as $\eps\to 0$}.
\end{equation}

Next, we show that
\begin{equation}\label{remainder-limit}
\Rem_\eps  \to 0\qquad\mbox{in $L^1(\Omega^r)\;\;$ as $\eps\to 0$}.
\end{equation}
From  \eqref{remainder-full} and \eqref{est-fX-gX-eps}
and the fact that $\Omega^r$ is a bounded domain, in order to prove \eqref{remainder-limit}, it suffices to show that
\begin{align}\label{remainder-limit-1}
&r^{(1)}_\eps \to 0 \qquad\qquad \mbox{in $L^2(\Omega^r)\;\;$ as $\eps\to 0$},\\
\label{remainder-limit-2}
&\curl\,\rr^{(2)}_\eps \to 0\qquad\, \mbox{in $L^1(\Omega^r)\;\;$ as $\eps\to 0$}.
\end{align}

To show \eqref{remainder-limit-1}, we first note that $\Omega^r\Subset\OmegaPM_r$
and $\rho\vv\in H^1(\OmegaPM_{r})\cap L^\infty(\OmegaPM_{r})$ by
Lemma \ref{approxVelocityLemma}\eqref{approxVelocityLemma-i1},
so that
$$
(\divg\,(\rho \vv))_\eps-\divg\,(\rho \vv) \to 0\qquad\mbox{in $L^2(\Omega^r)\;$ as $\eps\to 0$}.
$$
Thus, it remains to show that
\begin{equation}\label{remainder-limit-1-1}
\divg\,(\rho_\eps \vv_\eps)-\divg\,(\rho \vv) \to 0\qquad\mbox{in $L^2(\Omega^r)\;$ as $\eps\to 0$}.
\end{equation}
We first show that
$$
\divg\,(\rho_\eps \vv_\eps)-\divg\,(\rho \vv_\eps)\to 0\qquad\mbox{in $L^2(\Omega^r)\;$ as $\eps\to 0$}.
$$
Indeed,
$$
\divg\,(\rho_\eps \vv_\eps)-\divg\,(\rho \vv_\eps)
=\grad(\rho_\eps-\rho)\cdot\vv_\eps +(\rho_\eps-\rho)(\divg\, \vv_\eps-\divg\, \vv)+(\rho_\eps-\rho)\divg\, \vv.
$$
In the argument below, we use \eqref{conw-v-rho-omega}.
We see that $\grad(\rho_\eps-\rho)\cdot\vv_\eps\to 0$ in $L^2(\Omega^r)$ as $\eps\to 0$, because $\grad\rho_\eps\to\grad\rho$ in
$L^2(\Omega^r)$ and $\vv_\eps$ is uniformly bounded in $L^\infty(\Omega^r)$.
Also, $(\rho_\eps-\rho)(\divg\, \vv_\eps-\divg\, \vv)\to 0$ in $L^2(\Omega^r)$,
because $\rho-\rho_\eps$ is uniformly bounded in $L^\infty(\Omega^r)$ and $\divg\, \vv_\eps\to \divg\, \vv$ in $L^2(\Omega^r)$.
Finally,  $(\rho_\eps-\rho)\divg\, \vv\to 0$ in $L^2(\Omega^r)$, 
because $\rho-\rho_\eps$ is uniformly bounded in $L^\infty(\Omega^r)$
and converges to zero {\it a.e.}
in $\Omega^r$, and $\divg\,\vv\in L^2(\Omega^r)$, so that
$$
\int_{\Omega_r}(\rho_\eps-\rho)^2(\divg\, \vv)^2\,{\rm d}\xi \to 0 \qquad \mbox{as $\eps\to 0$},
$$
by the dominated convergence theorem.
The convergence:
$$
\divg\,(\rho \vv_\eps)-\divg\,(\rho \vv)\to 0\qquad\mbox{in $L^2(\Omega^r)\;\;$ as $\eps\to 0$}
$$
can be shown similarly. This completes the proof of \eqref{remainder-limit-1-1}, which leads to \eqref{remainder-limit-1}.

\smallskip
Now we show \eqref{remainder-limit-2}.
Note that
$$
\curl\,\rr^{(2)}_\eps=\partial_{\xi_1}\big((\vv\cdot\grad v_2)_\eps -\vv_\eps\cdot\grad (v_2)_\eps\big)
-\partial_{\xi_2}\big((\vv\cdot\grad v_1)_\eps -\vv_\eps\cdot\grad (v_1)_\eps\big).
$$
Then \eqref{remainder-limit-2} follows from Lemma \ref{commutatorLemma} in Appendix \ref{append1}, 
applied in $\Omega^{\rm ext}_r$,
with
$p=q=2$,
$b=v_j$, and $u=\partial_jv_k$ for the corresponding $j,k=1,2$, and $i=3-k$.

Combining the results above, \eqref{remainder-limit} is now proved.
Then, sending $\eps\to 0$ in \eqref{vorticityEq-weak-remainder} and using
\eqref{conv-fX-gX-eps}--\eqref{remainder-limit}, we obtain
\begin{equation*}
\int_\Omega \big(\rho f(X) \vv\cdot\grad\zeta +\rho g(X)\zeta\big)\,{\rm d}\xxi
 =0,
\end{equation*}
which is equivalent to \eqref{vorticityEq-weak}, by using \eqref{case-1-vortEq}.

\medskip
{\bf 4.3.} Next, we prove \eqref{vorticityEq-weak} in the case when the smooth function $\zeta$ satisfies:
\begin{equation}\label{case-2-vortEq}
 \zeta\equiv 0\qquad \mbox{in $\Omega\setminus \Nbhd_r(\GammaInt)\,\,$ for some $r\in (0, r_1)$.}
\end{equation}
Then, by \eqref{def-R1-1} and assumption \eqref{lower-reg-Th-i2} of Theorem \ref{lower-reg-Th},
equation \eqref{vorticityRho-Renorm-X} and the boundary conditions \eqref{bcOnWedgeSym} hold classically
on $\mbox{supp}(\zeta)\cap\Omega$ (for the equation, the argument is given from \eqref{vorticityEq-1} to \eqref{vorticityRho-Renorm}).
Then, multiplying
\eqref{vorticityRho-Renorm-X} by $\zeta$, integrating
over $\Omega$, and integrating by parts in the first term with the use of \eqref{bcOnWedgeSym},
we obtain \eqref{vorticityEq-weak} for $\zeta$ satisfying \eqref{case-2-vortEq}.

\medskip
{\bf 4.4.} Combining the two previous cases, we obtain \eqref{vorticityEq-weak} for all smooth $\zeta$ satisfying
\begin{equation}\label{case-3-vortEq}
 \zeta\equiv 0\qquad \mbox{in $\Omega\cap B_r(\PZer)\,\,$ for some $r>0$.}
\end{equation}
This in particular implies the following:
if $\PZer\notin \partial\Omega$, then
\eqref{vorticityEq-weak} holds for all smooth $\zeta$, since we can modify $\zeta$ outside $\Omega$ in this case
so that the modified function $\zeta$ satisfies \eqref{case-3-vortEq} for some $r>0$,
and this modification clearly does not affect \eqref{vorticityEq-weak} for $\zeta$.

Thus, the remaining proof is for the case that $\PZer\in \partial\Omega$. Note that this means that $\PZer$ is a
common endpoint of some of $\GammaExt_i$ and $\GammaExt_j$.

\medskip
{\bf 4.5}. Now we consider the  case that $\PZer$ is a common endpoint of some of $\GammaExt_i$ and $\GammaExt_j$, and fix $\zeta\in C^\infty(\bR^2)$.
Let $\psi\in C^\infty(\bR^2)$ be such that
$$
\mbox{$0\le\psi\le 1\,\,$ on $\bR^2$, $\qquad$ $\psi\equiv 0\,\,$ on $B_1$, $\qquad$
$\psi\equiv 1$ on $\bR^2\setminus B_2$.}
$$
 Let $\psi^r(\xxi)=\psi(\frac{\xxi}{r})$ for $r>0$. In particular,
\begin{equation}\label{grad-cutoff}
\psi^r\equiv 0 \;\;\mbox{in $B_r$},\qquad
 \psi^r\equiv 1 \;\;\mbox{in $\bR^2\setminus B_{2r}$},\qquad
|D\psi^r|\le \frac Cr, \qquad \mbox{supp}(D\psi^r)\subset B_{2r}.
\end{equation}
Here and below, the universal constant $C$ is independent of $r$, which may be different at different occurrence.
Then, for any small $r>0$, function $\zeta\psi^r$
satisfies \eqref{case-3-vortEq}, so \eqref{vorticityEq-weak} holds with the test function $\zeta\psi^r$ instead of $\zeta$.
Thus, we have
\begin{align}\label{vorticityEq-weak-r-ctoff}
&\int_\Omega \big( \rho f(X)(\vv\cdot\grad\zeta) +\rho g(X)\zeta\big)\psi^r\,{\rm d}\xxi
+
\int_\Omega  \rho f(X)(\vv\cdot\grad\psi^r) \zeta\,{\rm d}\xxi\nonumber\\
&\quad-
 \sum_{i=1}^\Nint\int_{\GammaInt_i}\rho f(X)  (\vv\cdot\nnu) \zeta \psi^r\,{\rm d}l=0.
\end{align}
We estimate the second integral in \eqref{vorticityEq-weak-r-ctoff}, by using \eqref{omega-over-rho-l2}, \eqref{grad-cutoff},
the boundedness of $\vv$ by
assumption \eqref{lower-reg-Th-i3},
and that $f'(X)$ is bounded on $\bR$:
$$
\bigg|  \int_\Omega \rho f(X) (\vv\cdot\grad\psi^r) \zeta\, {\rm d}\xxi \bigg|
\le\frac Cr \int_{\Omega\cap B_{2r}}(1+|X|)\,{\rm d}\xxi\le C\bigg( \int_{\Omega\cap B_{2r}} (1+|X|^2) \,{\rm d}\xxi  \bigg)^{\frac{1}{2}}\to 0\quad\,\,
\mbox{as }r\to 0.
$$
Notice that
$\psi^r\equiv 1$ in $\bR^2\setminus B_{2r}$ and $|\psi^r|\le 1$ in $B_{2r}$,
$(\rho,\vv)$ are bounded and $X\in L^2(\Omega)$, and $|(f(X), g(X))|\le C(1+|X|)$ from the assumptions of $f(X)$
and \eqref{f-g-vorticityRho-Renorm}.
Then, denoting by $L_r$ the difference between the first term in \eqref{vorticityEq-weak}
and the first term in \eqref{vorticityEq-weak-r-ctoff},
we have
$$
|L_r|=\bigg|\int_{\Omega\cap B_{2r}} \big( \rho f(X)(\vv\cdot\grad\zeta) +\rho g(X)\zeta\big)(1-\psi_r)\,{\rm d}\xxi \bigg|
\le C\int_{\Omega\cap B_{2r}}(1+|X|) \,{\rm d}\xxi\to 0\quad\,\,
\mbox{as } r\to 0.
$$
For the boundary integral, we obtain that, for each $i\in\{1, \cdots, \Nint\}$,
$$
\int_{\GammaInt_i}\rho f(X)  (\vv\cdot\nnu)\zeta \psi^r \,{\rm d}l
\rightarrow \int_{\GammaInt_i}\rho f(X)  (\vv\cdot\nnu)\zeta \,{\rm d}l
\qquad\,
\mbox{as } r\to 0.
$$
Indeed,  $\PZer\notin \overline{\GammaInt_i}$ by
Definition \ref{RiemannProblSolutStruct-1}\eqref{RiemannProblSolutStruct-1-i5},
so that  $\psi^r\equiv 1$ on $\GammaInt_i$ if $r<\frac 12\dist(\PZer, \GammaInt_i)$.

Combining the convergence facts shown above and sending $r\to 0$ in \eqref{vorticityEq-weak-r-ctoff},  we conclude \eqref{vorticityEq-weak}.

\medskip
{\bf 5.}
We now show that vorticity $\omega$ (and thus $X$) is smooth and not identically zero on $\GammaInt_1$.
Recall that $\GammaInt_1$ denotes a relatively open curve segment and $\GammaInt_1$ is a shock
by  Definition \ref{RiemannProblSolutStruct-1}\eqref{RiemannProblSolutStruct-1-i6}.

\begin{lemma}\label{vorticityOnShockLemma}
Under the assumptions of {\rm Theorem \ref{lower-reg-Th}},
vorticity $\omega$
on $\GammaInt_1$ from the $\Omega$--side is
continuous and not identically zero on $\GammaInt_1$, and is bounded on $\overline{\GammaInt_1}$.
\end{lemma}

This can be proved as follows:
First, it follows directly from the regularity of
$\GammaInt_1$ in Definition \ref{RiemannProblSolutStruct-1}\eqref{RiemannProblSolutStruct-1-i5}
and assumption
\eqref{lower-reg-Th-i2} of Theorem \ref{lower-reg-Th}
that $\omega$
on $\GammaInt_1$ from the $\Omega$--side is continuous on $\GammaInt_1$
and bounded on $\overline{\GammaInt_1}$.
Then, in the rest of the proof, its suffices to show that $\omega$ is not identically zero on $\GammaInt_1$.

Since equations \eqref{selfSimEuler}--\eqref{selfSimEuler-moment} and the Rankine-Hugoniot conditions \eqref{RH}
are invariant under the coordinate rotation and translation,
at any fixed point $P\in\GammaInt_1$, we can choose the coordinates $\xxi$
such that the $\xi_1$--direction is tangent
to $\GammaInt_1$ at $P$.
Then $\GammaInt_1$ is the graph of a function $f_s$ locally, \emph{i.e.}, $\GammaInt_1=\{\xxi\,:\, \xi_2=f_s(\xi_1)\}$
locally near $P$.
Thus, $f_s$ is in $C^2$ in a neighborhood of $P$ by
Definition \ref{RiemannProblSolutStruct-1}\eqref{RiemannProblSolutStruct-1-i5}, and we can obtain that $f'_s=0$ at point $P$ by
choosing the appropriate coordinate system.
Note that $(f'_s,-1)$ is a normal of $\GammaInt_1$, so that the Rankine-Hugoniot conditions \eqref{RH} become
\begin{equation}\label{RHf}
\begin{cases}
(\rho v_1-\rho_1 v_1^-)f_s'-(\rho v_2-\rho_1 v_2^-)=0,\\[1mm]
(\rho v_1^2+p-\rho_1(v_1^-)^2-p_1)f'_s-(\rho v_1v_2-\rho_1v_1^-v_2^-)=0,\\[1mm]
(\rho v_1v_2-\rho_1v_1^-v_2^-)f_s'-(\rho v_2^2+p-\rho_1 (v_2^-)^2-p_1)=0,
\end{cases}
\end{equation}
where $\vv^-=(v_1^-,v_2^-)=(u_1^--\xi_1, u_2^--\xi_2)$ with some constants $u_1^-$ and $u_2^-$, and $\vv=(v_1,v_2)$.
Taking the tangential derivative $\partial_{\bt}:=\partial_{\xi_1}+f_s'\partial_{\xi_2}$
of the Rankine-Hugoniot conditions \eqref{RHf} along $\GammaInt_1$ and using the condition that $f_s'=0$ at $P$,
we obtain that, at point $P$,
\begin{equation*}
\begin{cases}
(\rho v_1-\rho_1 v_1^-)f_s''-(\rho v_2)_{\xi_1}=0,\\[1mm]
(\rho v_1^2+p-\rho_1(v_1^-)^2-p_1)f''_s-(\rho v_1v_2)_{\xi_1}-\rho_1v_2^-=0,\\[1mm]
(\rho v_2^2)_{\xi_1}+c^2\rho_{\xi_1}=0.
\end{cases}
\end{equation*}

By equations \eqref{selfSimEuler}--\eqref{selfSimEuler-moment} and the definition: $\omega=(v_1)_{\xi_2}-(v_2)_{\xi_1}$,
it follows from a straightforward but long calculation that
\begin{align*}
(v_1)_{\xi_1}=&-1-\frac{c^2+v_2^2}{\rho|\vv|^2}v_1\rho_{\xi_1}+\frac{c^2-v_2^2}{\rho|\vv|^2}v_2\rho_{\xi_2}-\frac{v_1v_2}{|\vv|^2}\omega,\\
(v_1)_{\xi_2}=&-\frac{c^2-v_1^2}{\rho|\vv|^2}v_2\rho_{\xi_1}-\frac{c^2-v_2^2}{\rho|\vv|^2}v_1\rho_{\xi_2}+\frac{v_1^2}{|\vv|^2}\omega,\\
(v_2)_{\xi_1}=&-\frac{c^2-v_1^2}{\rho|\vv|^2}v_2\rho_{\xi_1}-\frac{c^2-v_2^2}{\rho|\vv|^2}v_1\rho_{\xi_2}-\frac{v_2
^2}{|\vv|^2}\omega,\\
(v_2)_{\xi_2}=&-1+\frac{c^2-v_1^2}{\rho|\vv|^2}v_1\rho_{\xi_1}-\frac{c^2+v_1^2}{\rho|\vv|^2}v_2\rho_{\xi_2}+\frac{v_1v_2}{|\vv|^2}\omega.
\end{align*}
Therefore, by a long computation, we obtain that $(\rho_{\xi_1}, \rho_{\xi_2},\omega)$ satisfy the linear system at $P$:
\begin{equation}\label{linSystOmega-Shock}
\begin{pmatrix}
a_1 & b_1 & c_1\\
a_2 & b_2 & c_2\\
a_3 & b_3 & c_3
\end{pmatrix}
\begin{pmatrix}
\rho_{\xi_1} \\ \rho_{\xi_2} \\ \omega
\end{pmatrix}
=
\begin{pmatrix}
d_1f''_s \\ d_2f''_s \\ 0
\end{pmatrix},
\end{equation}
where
\begin{align*}
&(a_1, b_1,c_1)=((c^2-2v_1^2-v_2^2)v_2,\, (c^2-v_2^2)v_1,\, \rho v_2^2),\\
&(a_2, b_2,c_2)=(2(c^2-v_1^2)v_1v_2,\,-(c^2-v_2^2)(v_2^2-v_1^2),\, 2\rho v_1 v_2^2),\\
& (a_3, b_3,c_3)=(3v_1^2v_2^2+v_1^2c^2+v_2^4-v_2^2c^2,\, -2(c^2-v_2^2)v_1v_2,\, -2\rho v_2^3),
\end{align*}
and
\begin{equation*}
d_1=-|\vv|^2\big(\rho v_1-\rho_1v_1^-\big),\qquad d_2=-|\vv|^2\big(\rho v_1^2+p-\rho_1 (v_1^-)^2-p_1\big).
\end{equation*}
Notice that
\begin{equation*}
\begin{vmatrix}
a_1 & b_1 & c_1\\
a_2 & b_2 & c_2\\
a_3 & b_3 & c_3
\end{vmatrix}
=\rho v_2 (c^2-v_2^2)^2|\vv|^4\neq0,
\end{equation*}
where we have used the fact that  $v_2=\vv\cdot\bn\neq0$ at $P$ by the entropy condition and the ellipticity assumption that $|\vv|<c$ on $\GammaInt_1$
in Theorem 3.2.
Therefore, from \eqref{linSystOmega-Shock},
\begin{equation*}
\omega=
\frac
{d_1
\begin{vmatrix}
a_2 & b_2\\
a_3 & b_3
\end{vmatrix}
-d_2
\begin{vmatrix}
a_1 & b_1\\
a_3 & b_3
\end{vmatrix}}
{\rho v_2 (c^2-v_2^2)^2|\vv|^4}f_s''.
\end{equation*}
Since $f'_s=0$ at $P$, it follows from \eqref{RHf} that
\begin{equation}\label{RHP}
\rho v_2-\rho_1 v_2^-=0,\,\,\,\, v_1=v_1^-, \,\,\,\,\rho v_2^2+p-\rho_1 (v_2^-)^2-p_1=0\qquad\,\,\mbox{at }P.
\end{equation}
Then
\begin{equation}\label{vorticityOnShockExpression}
\omega=\frac{v_1\big((\rho-\rho_1)v_2^2+(p-p_1)\big)}{\rho v_2}f_s''.
\end{equation}
Notice that
$\vv\cdot\nnu\le -C^{-1}$  on $\GammaInt_{1}$ by Definition \ref{RiemannProblSolutStruct-1}\eqref{RiemannProblSolutStruct-1-i6},
where $\vv$ on $\GammaInt_{1}$ is taken from the $\Omega$--side, and $\nnu$ is the outer normal with respect to $\Omega$.
By \eqref{shock-RHp}, this implies that
$\vv\cdot\nnu<0$ on $\GammaInt_{1}$ for $\vv$ taken from the $\Lambda\setminus\Omega$--side.
Using the  entropy condition \eqref{shock-entropy-RHp} on $\GammaInt_{1}$,
we obtain that  $\rho>\rho_1$ and $p>p_1$ on $\GammaInt_1$.
Moreover, $v_1(\hat P)=(\vv\cdot\ttau)(\hat P)\ne 0$
for point $\hat P\in \GammaInt_1$ specified in
Definition \ref{RiemannProblSolutStruct-1}(\ref{RiemannProblSolutStruct-1-i7}).
Therefore, $\omega\neq0$ at $\hat P$. Then the vorticity is not identically zero on
$\GammaInt_1$.

\medskip
{\bf 6.}
Using Lemma \ref{vorticityEqnProp}, we can formally choose $f(s)=s^2$ with $g(s)=0$
by \eqref{f-g-vorticityRho-Renorm}, and use $\zeta\equiv 1$ to obtain that, by \eqref{vorticityEq-weak}.
$$
0=
-\int_{\Shock} \big|\frac{\omega}{\rho} \big|^2 \rho \vv\cdot \nnu \,{\rm d}l>0.
$$
The strict inequality above follows from Definition \ref{RiemannProblSolutStruct-1}\eqref{RiemannProblSolutStruct-1-i6}
and Lemma \ref{vorticityOnShockLemma}.
This (formally) shows that assumption \eqref{v-rho-l2}
leads to a contradiction.
A minor technical point is that function $f(s)=s^2$ does not satisfy the assumption of Lemma \ref{vorticityEqnProp}:
$\|f'\|_{C^{0,1}(\bR)}<\infty$.

To make this rigorously, we approximate
$f(s)=s^2$ by the functions that satisfy the assumptions of Lemma \ref{vorticityEqnProp}
and verify the limit process
of this approximation.
More specifically, for any $M>1$, define $f_M\in C^1(\bR)$ by $f_M(0)=0$ and $f'_M(t)=2\min\{|t|, M\}\sign(t)$.
Then
\[
f_M(t)=
\begin{cases}
t^2\qquad&\mbox{if }|t|\leq M,\\
M^2+2M(|t|-M)\qquad&\mbox{if }|t|> M.
\end{cases}
\]
It follows that $\|f'_M\|_{C^{0,1}(\bR)}<\infty$
 and the function
defined by \eqref{f-g-vorticityRho-Renorm} is
\begin{equation}\label{defGM}
g_M(t)=
\begin{cases}
0&\mbox{if }|t|\leq M,\\
2(M^2-M|t|)\qquad&\mbox{if }|t|> M.
\end{cases}
\end{equation}
Now \eqref{vorticityEq-weak} holds with $f_M$, $g_M$, and any $\zeta\in C^\infty(\bR^2)$.
Choosing $\zeta\equiv1$, then we have
\begin{equation}\label{5.29}
\int_{\Omega}\rho g_M(X)\,{\mathrm{d}}\xxi=\int_{\GammaInt}\rho f_M(X)\, (\vv\cdot\nnu) \,{\rm d}l.
\end{equation}
We send $M\to\infty$. Clearly, $g_M(X)\to 0$ pointwise in $\Omega$.
Also, from \eqref{defGM}, we obtain
$$
|g_M(t)|\le 2t^2 \qquad \mbox{ for all $M>1$ and $t\in\bR$}.
$$
Thus, using that $X\in L^2(\Omega)$ and $\rho\in L^\infty(\Omega)$, we have
$$
\int_{\Omega}\rho g_M(X)\,{\mathrm{d}}\xxi\to 0\qquad\mbox{as $M\to\infty$},
$$
by the dominated convergence theorem.
Also, since $X\in L^\infty(\GammaInt)$ by assumptions \eqref{lower-reg-Th-i2}--\eqref{lower-reg-Th-i3},
then it follows from the explicit form of $f_M$ that
$$
\int_{\GammaInt}\rho f_M(X)\,(\vv\cdot\nnu) \,{\rm d}l=
\int_{\GammaInt}\rho |X|^2\,(\vv\cdot\nnu) \,{\rm d}l
\qquad\mbox{for all $M\ge \|X\|_{L^\infty(\GammaInt)}$}.
$$
Thus, sending $M\to\infty$ in \eqref{5.29}, we obtain
$$
\int_{\GammaInt}\rho |X|^2\,(\vv\cdot\nnu) \,{\rm d}l=0.
$$
Since  $\rho\in [C_0^{-1},\,C_0]$
by assumption \eqref{lower-reg-Th-i3}
and $\vv\cdot\nnu\le 0$ on $\GammaInt_i$, for $i=2,\cdots, \Nint$,
from  Definition \ref{RiemannProblSolutStruct-1}\eqref{RiemannProblSolutStruct-1-i6},
we obtain
$$
\int_{\GammaInt_1}\rho |X|^2\,(\vv\cdot\nnu) \,{\rm d}l\ge 0.
$$
This is a contradiction since $\rho\in [C_0^{-1},\,C_0]$,
$\vv\cdot\nnu\le -C^{-1}$ on $\GammaInt_1$ by Definition \ref{RiemannProblSolutStruct-1}\eqref{RiemannProblSolutStruct-1-i6},
and $\omega$ (and thus $X$) is continuous and not identically zero on $\GammaInt_1$ by Lemma \ref{vorticityOnShockLemma}.

This completes the proof of the fact that, under the assumptions of the theorem, it is not possible that 
\eqref{v-rho-l2} holds.

\smallskip
{\bf 7.} 
We now complete the proof as claimed.
On the contrary, assume that
\begin{equation}\label{v-l2-x}
\vv\in H^1(\Omega).
\end{equation}
Then, combining \eqref{v-l2-x} with the assumptions in the theorem, we have
\begin{equation}\label{propStep7-ThProof}
\vv\in H^1(\Omega)\cap L^\infty(\Omega), \quad\,\,
C^{-1}\le \rho\le C, \quad\,\,
{\tilde C}^{-1}\le p\le \tilde C,
\end{equation}
where $C,\; \tilde C\ \ge 1$, and we recall that $p=\frac{\rho^{\gamma}}{\gamma}$.
We show that, with these properties,  we can obtain from the weak form of 
the Euler system \eqref{weakSolEuler-Equal-cmass}--\eqref{weakSolEuler-Equal-cmoment} that
\begin{align}
\label{selfSimEuler-moment-mod-weak}
 \int_\Omega \big((\rho\vv\cdot \grad \vv+\rho\vv)\cdot\zeta- p\,\divg \zeta\big)\, {\rm d}\xxi=0,
\end{align}
for any $\zeta=(\zeta_1, \zeta_2)\in C^1_{\rm c}(\Omega; \bR^2)$.
To show \eqref{selfSimEuler-moment-mod-weak}, we compute for $j=1,2$, using \eqref{propStep7-ThProof}:
\begin{equation}\label{calc1Step7-ThProof}
\begin{split}
  \int_\Omega \rho \sum_{i=1}^2v_i  v_j \partial_i\zeta_j \, {\rm d}\xxi & =
  \int_\Omega \rho \sum_{i=1}^2v_i \partial_i( v_j \zeta_j )\, {\rm d}\xxi  - \int_\Omega \rho \sum_{i=1}^2v_i  (\partial_i v_j) \zeta_j  \, {\rm d}\xxi \\
   & = \int_\Omega \rho \vv\cdot \grad( v_j \zeta_j )\, {\rm d}\xxi  - \int_\Omega \rho (\vv\cdot\grad v_j) \zeta_j  \, {\rm d}\xxi.
\end{split}
\end{equation}
Next, we note that, by the properties in \eqref{propStep7-ThProof}, we can use the test function $\phi=v_j\zeta_j$
in \eqref{weakSolEuler-Equal-cmass} to obtain
$$
\int_\Lambda \rho\vv\cdot\grad (v_j\zeta_j)\, {\rm d}\xxi = \int_\Omega 2\rho v_j\zeta_j\, {\rm d}\xxi.
$$
Substituting this into the last expression of \eqref{calc1Step7-ThProof} for $j=1,2$, we have
$$
\int_\Omega \rho\vv\otimes \vv: D\zeta\, {\rm d}\xxi= \int_\Omega \big(2\rho v-\rho(\vv\cdot\grad) \vv\big)\cdot \zeta\, {\rm d}\xxi.
$$
Substituting this into \eqref{weakSolEuler-Equal-cmoment}, we obtain \eqref{selfSimEuler-moment-mod-weak}.

From
\eqref{selfSimEuler-moment-mod-weak} and the fact that
$\rho\vv\cdot \grad \vv+\rho\vv \in L^2(\Omega)$ by \eqref{propStep7-ThProof}, 
it follows from \eqref{selfSimEuler-moment-mod-weak}
that the weak derivative of $p$ exists in $\Omega$ and that $\nabla p \in L^2(\Omega)$. 
Then, using \eqref{propStep7-ThProof}, we have   
$$p\in H^1(\Omega)\cap L^\infty(\Omega).
$$
From this, since $\rho=(\gamma p)^{1/\gamma}$ and ${\tilde C}^{-1}\le p\le \tilde C$,
we obtain 
$$\rho\in H^1(\Omega)\cap L^\infty(\Omega).
$$
Combining this with \eqref{v-l2-x}, we obtain that $(\rho, \vv)\in  H^1(\Omega)$,
which contradicts \eqref{v-rho-l2-b}.
Then we conclude that \eqref{v-l2-x} is false.
\end{proof}
%%%%%%%%%%%%%%%%%%%%%%%%%%%%%%%%%%%%%%%%%%%%%%%%%%%%%%%%%%%%%%%%%%%%%%%%%%%%%%%%%%%%%%%%%%%
%%%%%%%%%%%%%%%%%%%%%%%%%%%%%%%%%%%%%%%%%%%%%%%%%%%%%%%%%%%%%%%%%%%%%%%%%%%%%%%%%%%%%%%%%%

%%%%%%%%%%%%%%%%%%%%%%%%%%%%%%%%%%%%%%%%%%%%%%%%%%%%%%%%%%%%%%%%%%%%%%%%%%%%%%%%%%%%%%%%%%%%%%
%%%%%%%%%%%%%%%%%%%%%%%%%%%%%%%%%%%%%%%%%%%%%%%%%%%%%%%%%%%%%%%%%%%%%%%%%%%%%%%%%%%%%%%%%%%%%%
\section{Applications to Transonic Shock problems}
\label{SectionApplications}
In this section, we employ the general framework of Definition \ref{RiemannProblSolutStruct-1}
and the main theorem, Theorem \ref{lower-reg-Th},
established in \S 3 to analyze the low regularity of entropy solutions of
several transonic shock problems including the regular shock reflection problem,
the Prandtl reflection problem, the Lighthill diffraction problem, and the Riemann problem with four-shock interactions.

First of all, in verifying the conditions of Definition \ref{RiemannProblSolutStruct-1},
condition \eqref{RiemannProblSolutStruct-1-i7} is often difficult to be verified directly.
The following lemma is useful for that; in fact, it is used in all the applications we describe below.

\begin{lemma}\label{PtWithNonzeroTangVelocCurv-lemma}
Let $(\rho,\vv)$ be a solution of the Riemann problem that satisfies conditions
\eqref{RiemannProblSolutStruct-1-i1}--\eqref{RiemannProblSolutStruct-1-i6} of
{\rm Definition \ref{RiemannProblSolutStruct-1}} and assumption \eqref{lower-reg-Th-i2} of {\rm Theorem \ref{lower-reg-Th}}.
Assume  that
\begin{enumerate}[\rm (a)]
\item \label{PtWithNonzeroTangVelocCurv-lemma-i1}
$\GammaInt_1=\partial\Omega\cap\partial\Lambda_j$ for some $j\in\{1, \cdots, M\}$, where
 $\Lambda_j$ are defined in {\rm Definition \ref{RiemannProblSolutStruct-1}\eqref{RiemannProblSolutStruct-1-i1}}.
\item \label{PtWithNonzeroTangVelocCurv-lemma-i2}
$\GammaInt_1$ is not a segment of straight line. In particular,
denoting by $P_1$ and $P_2$ the endpoints of $\GammaInt_1$,
there exists a point $P^* \in \overline{\Gamma_1^{\rm int} } \setminus \{ P_1\}$ such that $\bt(P_1)\ne\pm\bt(P^*)$,
where $\bt(\cdot)$ is a unit tangent vector to $\GammaInt_1$ at a point.
\item \label{PtWithNonzeroTangVelocCurv-lemma-i3}
$(\vv\cdot\bt)(P_1)\ne 0$.
\end{enumerate}
Then condition \eqref{RiemannProblSolutStruct-1-i7} of
{\rm Definition \ref{RiemannProblSolutStruct-1}} is satisfied.
\end{lemma}

\begin{proof}
\newcommand{\SOne}{S}
By re-indexing sets $\Lambda_j$, we can assume that $\GammaInt_1=\partial\Omega\cap\partial\Lambda_1$.
Then $(\rho_1, \vv_1)$ is the uniform state in $\Lambda_1$, where $\rho_1$ is constant and
$\vv_1=(u_1^{(1)}, u_2^{(1)})-\xxi$ with a constant state $(u_1^{(1)}, u_2^{(1)})$. This state is called state $(1)$.

We show the existence of a point $\hat P$ in the relative interior of $\GammaInt_{1}$ such that the
curvature of $\GammaInt_{1}$ is non-zero at $\hat P$ and $(\vv\cdot\bt)(\hat P)\ne 0$.
Denote by  $\SOne$ the line tangential to $\overline{\GammaInt_1}$ at $P_1$.

Denote by $Q$ the intersection point of line $\SOne$ and line $L$ through center $O_1=(u_1^{(1)}, u_2^{(1)})$ of state $(1)$
perpendicular to $\SOne$. Note that,
for any point $P$, $\vv_1=O_1-P$ so that $(\vv_1\cdot\bt_{S})(Q)=0$ and $(\vv_1\cdot\bt_{S})(P)\neq0$
for all $P\in \SOne\setminus{Q}$, where we recall that $\vv_1$ is the pseudo-velocity of the uniform state
in $\Lambda_1$ which also defines $\vv_1$ on the whole $\bR^2$.

Denote by $\hat Q$ the point on $\SOne$ such that $\GammaInt_1$ coincides with $\SOne$ between points $P_1$ and $\hat Q$, but not on any larger interval
extended through $\hat Q$.
Note that it is possible that $\hat Q=P_1$,
but $\hat Q\ne P^*$ since the tangential line to $\overline{\GammaInt_1}$ at $P^*$ is not parallel to $L$
by our condition \eqref{PtWithNonzeroTangVelocCurv-lemma-i2}.

If $\hat Q\ne Q$, then $(\vv\cdot\bt)(\hat Q)=(\vv_1\cdot\bt)(\hat Q)\ne 0$,
where $\vv$ is the velocity on $\GammaInt_1$ from the $\Omega$--side and we have used \eqref{shock-RHp}.
Also, from the definition of $\hat Q$, in any neighborhood of $\hat Q$, there exists a point $P\in \GammaInt_1$ with nonzero curvature.
Thus, it follows from the $C^2$--regularity of $\GammaInt_1$ in its relative interior
(by condition \eqref{RiemannProblSolutStruct-1-i5} of Definition \ref{RiemannProblSolutStruct-1})
and the continuity of $\vv$ in $\Omega$ near and up to $\GammaInt_1$ (by
assumption \eqref{lower-reg-Th-i2} of Theorem \ref{lower-reg-Th})
that there exists a point $\hat P\in (\GammaInt_1)^0$ with non-zero curvature and
$(\vv\cdot\bt)(\hat P)\ne 0$.

Therefore, the remaining possibility is that $\hat Q=Q$.
Note that $(\vv_1\cdot\bt_{S})(Q)=0$.
Moreover, using condition \eqref{PtWithNonzeroTangVelocCurv-lemma-i3},
the Rankine-Hugoniot conditions, and the regularity of $\GammaInt_1$ and $(\rho, \vv)$
given in condition \eqref{RiemannProblSolutStruct-1-i5} of Definition \ref{RiemannProblSolutStruct-1}
and assumption \eqref{lower-reg-Th-i2} of Theorem \ref{lower-reg-Th},
we see that $(\vv_1\cdot\bt_{S})(P_1)=(\vv\cdot\bt_{S})(P_1)\ne 0$.
This implies that $Q\ne P_1$.
Also, since $\hat Q\ne P^*$ as we discussed above, then $Q\ne P^*$ for the present case.
Thus, for the present case $Q=\hat Q$, it follows that $Q$ is
an interior point of  $\GammaInt_1$,
and the part of $\GammaInt_1$ between $P_1$ and $Q$ lies on the straight line $\SOne$.
In particular, line $\SOne$ is tangential to $\GammaInt_1$ at $Q$.
We now shift and rotate the coordinates to have the origin at $Q$ and
the coordinates $\xi_1$ and $\xi_2$ to be along $\SOne$ and $L$, respectively.
Then  $O_1=(0, \tilde v_1)$ for some $\tilde v_1\in\bR$ (in fact, $\tilde v_1\ne 0$ by condition
    \eqref{RiemannProblSolutStruct-1-i5} of Definition \ref{RiemannProblSolutStruct-1}, but this
    will not be used below).
To fix notation, let the $\xi_1$--axis
along $\SOne$ be oriented so that $P_1=(\xi_{P_1}, 0)$ with $\xi_{P_1}<0$.
Since $\SOne=\{\xi_2=0\}$ is tangential to $\GammaInt_1$ at $Q=(0,0)$,
then curve $\GammaInt_1\cap B_r(Q)$ is a graph for some $r>0$:
There exists $f\in C^2(\bR)$ such that
$$
\GammaInt_1\cap B_r(Q)=\big\{(\xi_1, f(\xi_1))\;:\;\xi_1\in (a,b)\big\}\qquad
\mbox{for some $a<0$, $b>0$, and $f\equiv 0$ on $(a,0)$},
$$
where the last assertion holds because $\GammaInt_1$ lies on $\SOne$ between $P_1$ and $Q$.
Thus, $f'(0)=f''(0)=0$ so that $|f'(\xi_1)|\le O(\eps)\xi_1$ for all $\xi_1\in (-\eps, \eps)$,
where $O(\eps)\to 0$ as $\eps\to 0^+$.
For any $P=(\xi_1, f(\xi_1))\in \GammaInt_1\cap B_r(Q)$, we have
   $$
   \vv_1(P)=O_1-P=(-\xi_1, \,\tilde v_2-f(\xi_1)), \qquad
   \bt(P)=\frac{(1,\,f'(\xi_1))}{\sqrt{1+(f'(\xi_1))^2}}.
   $$
Then
$$(
   \vv_1\cdot\bt)(P)=\frac {-\xi_1+(\tilde v_2-f(\xi_1))f'(\xi_1)}{\sqrt{1+(f'(\xi_1))^2}}
   =\frac {-\xi_1+O(\eps)\xi_1}{\sqrt{1+(f'(\xi_1))^2}}\ne 0\qquad\mbox{for $\xi_1\in (0, \eps)$ if $\eps$ is small.}
$$
Thus, $\vv\cdot\bt=\vv_1\cdot\bt\ne 0$ at any $P=(\xi_1, f(\xi_1))$ with $\xi_1\in (0, \eps)$.
Since $\hat Q=Q$, {\it i.e.},  for every $\eps>0$,
there exists a point $\xi_1\in (0, \eps)$ such that $\GammaInt_1$ has non-zero
curvature at $P=(\xi_1,  f(\xi_1))$,
it follows that there exists a point $\hat P$, at which the tangential velocity and the curvature of $\GammaInt_1$ are nonzero,
in the present case. This completes the proof.
\end{proof}

%%%%%%%%%%%%%%%%%%%%%%%%%%%%%%%%%%%%%%%%%%%%%%%%%%%%%%%%%%%%%%%%%%%%%%%%%%%%%%%%%%
\subsection{Lower Regularity of the Regular Shock Reflection Solutions for the Isentropic Euler System}
The first Riemann problem we address is
the regular shock reflection problem for the isentropic Euler system \eqref{isentropisEulersystem}.
When a plane incident shock $S_0:=\Gamma_{\rm shock}^0$ hits a two-dimensional wedge,
a shock reflection-diffraction configuration takes shape.
The incident shock $S_0$ separates two constant states: state $(0)$
with velocity $\uu^{(0)}=(0,0)$ and density $\rho_0$ ahead of the shock, and
state (1) with velocity $\uu^{(1)}=(u^{(1)}_1,0)$ and density $\rho_1$
behind the shock, where $\rho_1>\rho_0$, and ${u}^{(1)}_1>0$ is
determined by $(\rho_0, \rho_1, \gamma)$ through the Rankine-Hugoniot conditions on $S_0$.
The incident shock $S_0$ moves in the direction of the $x_1$--axis and hits the wedge vertex at the initial time.
The slip boundary
condition $\uu\cdot\bn=0$ is prescribed on the wedge boundary,
where $\uu$ is the velocity of gas.
Since state (1) does not satisfy the boundary condition,
the shock reflection-diffraction configuration
occurs at later time, which is self-similar.
Depending on the
flow parameters and the wedge angle, there may be various patterns of shock reflection-diffraction
configurations,
including Regular Reflection and Mach Reflection.

The regular reflection problem is a lateral Riemann problem in the region
$$
\Lambda=\bR^2_+\setminus\big\{\mathbf{x}\; : \; x_1>0, \, 0<x_2<x_1\tan\theta_{\rm w}\big\},
$$
where $\bR^2_+=\bR^2\cap\{x_1>0\}$.
We seek functions $(\rho, \uu)(\mathbf{x}, t)$
satisfying  system \eqref{isentropisEulersystem} in $\Lambda$
with the boundary condition $\uu\cdot\nnu=0$ on $\partial\Lambda$ and the initial data:
\begin{equation*}
(\rho, \uu)(\mathbf{x}, 0)=
\begin{cases}
 (\rho_0, \uu^{(0)})\qquad  &\mbox{if }\; \mathbf{x}\in\Lambda\cap\{x_1<0\}, \\
 (\rho_1, \uu^{(1)})\qquad\, &\mbox{if }\; \mathbf{x}\in\Lambda\cap\{x_1>0\}.
\end{cases}
\end{equation*}

This initial-boundary value problem is invariant under scaling \eqref{urho-self-similar},
so we seek self-similar solutions $(\rho, \vv)=(\rho, \vv)(\xxi)$, where the self-similar variables $\xxi$
and the pseudo-velocity
$\vv=\uu-\xxi$ are introduced in \S \ref{Euler-Eq-subseq}.

\begin{figure}[htp]
\begin{center}
	\begin{minipage}{0.51\textwidth}
		\centering
		\includegraphics[width=0.6\textwidth]{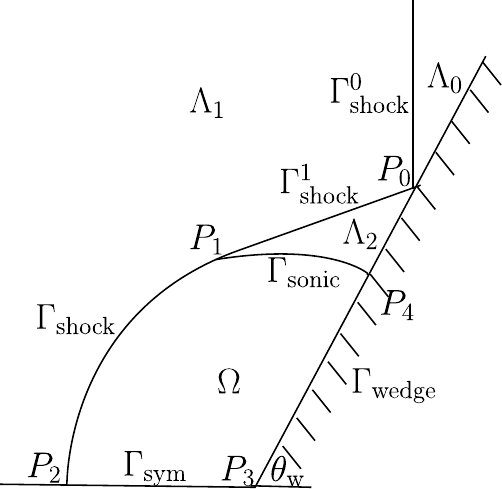}
\caption{Supersonic regular reflection}
\label{figure:shock refelction}
		\end{minipage}
	\hspace{-0.25in}
	\begin{minipage}{0.51\textwidth}
		\centering
		\vspace{5mm}
       \includegraphics[width=0.55\textwidth]{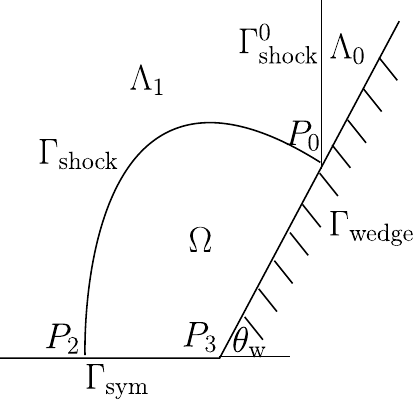}
\caption{Subsonic regular reflection}
\label{figure:shock refelction-subs}
			\end{minipage}
\end{center}
\end{figure}

First, consider the problem with an assumption on the symmetry with respect to
the $x_1$--axis.
Then we can consider only the upper half-plane $\{x_2>0\}$
and prescribe the slip boundary condition $\uu\cdot\nnu=0$
on the symmetry line $\{x_2=0\}$, so that there is only one reflection point $P_0$ to
be considered (see Figs \ref{figure:shock refelction}--\ref{figure:shock refelction-subs}).

The regular shock reflection-diffraction configuration is characterized
by the fact that the reflection occurs at point $P_0$
of the intersection of the incident shock with the wedge boundary.
Figs.  \ref{figure:shock refelction}--\ref{figure:shock refelction-subs}
show the structure of regular shock
reflection-diffraction configurations in self-similar coordinates.

A necessary condition for the existence of a regular reflection-diffraction
configuration is the existence of
the constant state $(2)$ and the reflected shock line such that state $(2)$ satisfies both
the slip boundary condition on the wedge and the Rankine-Hugoniot conditions
\eqref{RH}--\eqref{shock-RHp} on the shock
with state $(1)$ across the reflected shock line written at $P_0$.
These conditions lead to a system of algebraic equations for the constant velocity and density of state $(2)$.
Moreover, the entropy condition  \eqref{weakSolEuler-Equal-entr}
becomes an inequality in terms of the parameters
of states $(1)$ and $(2)$.

It is well-known (see {\it e.g.} \cite[Chapter 18]{CFBOOK2018} for the full Euler system case;
the argument for the isentropic Euler system case is similar) that,
given the parameters of states $(0)$ and $(1)$, there
exists a detachment angle $\theta_{\rm w}^{\rm d}\in (0, \frac{\pi}{2})$ such
that the system of algebraic equations for parameters of state (2) has two solutions for each wedge angle
$\theta_{\rm w}\in (\theta_{\rm w}^{\rm d}, \frac \pi 2)$ such that the entropy condition \eqref{weakSolEuler-Equal-entr}
is satisfied for the resulting two-shock configuration. These two solutions become
equal when $\theta_{\rm w}=\theta_{\rm w}^{\rm d}$.
Thus, two types of two-shock configurations occur at $P_0$ for each
$\theta_{\rm w}\in (\theta_{\rm w}^{\rm d}, \frac\pi 2)$.
For such $\theta_{\rm w}$,
state $(2)$ with the smaller density is called a weak state $(2)$.
It is expected that the weak state $(2)$ is physical,
while the strong state $(2)$ is not stable as the wedge angle tends to $\frac{\pi}{2}$
(as shown
in \cite{CFANN2010} for the potential flow case).
In the case of the potential flow model, the global existence
of regular shock reflection solutions for
all $\theta_{\rm w}\in (\theta_{\rm w}^{\rm d}, \frac{\pi}{2})$
with $(\rho, \uu)$ at $P_0$ determined by the weak states $(2)$
has been established in
\cite{CFANN2010,CFBOOK2018}. For the full or isentropic Euler system,
the existence of regular reflection solutions
is an outstanding open problem.
From now on, state $(2)$ always refers to the weak state $(2)$, which is unique for each
$\theta_{\rm w}\in (\theta_{\rm w}^{\rm d}, \frac{\pi}{2})$.

Furthermore, it follows from the Rankine-Hugoniot conditions \eqref{RH}--\eqref{shock-RHp} on the
straight shock $S_1:=\Gamma_{\rm shock}^1$ between states $(1)$ and $(2)$,
the slip boundary condition on the wedge for state $(2)$,
$\vv_1(\xxi)=(u_1^{(1)}, 0)-\xxi$, and $\vv_2(\xxi)=\uu_2-\xxi$ that
\begin{equation}\label{S1noteVertical}
 \mbox{\it The shock line $S_1$ between states $(1)$ and $(2)$
is not vertical for all $\theta_{\rm w}\in (\theta_{\rm w}^{\rm d}, \frac\pi 2)$. }
\end{equation}
Moreover, from the entropy condition \eqref{shock-entropy-RHp} on $S_1$, we have
\begin{equation}\label{rho2LgRho_1}
\rho_2>\rho_1.
\end{equation}

Depending on the wedge angle,  state $(2)$ can be either supersonic or
subsonic at $P_0$, {\it i.e.}, either $|\vv_2(P_0)|>c_2$ or the opposite inequality holds,
where $c_2=\rho_2^\frac{\gamma-1}2$ is the (constant) speed of sound of state (2).
Moreover, for $\theta_{\rm w}$ near $\frac\pi 2$
(resp. for $\theta_{\rm w}$ near $\theta_{\rm w}^{\rm d}$),
 state $(2)$ is supersonic
(resp. subsonic) at $P_0$.
The type of state $(2)$ at $P_0$ for a given wedge
angle $\theta_{\rm w}$  determines the type of reflection, supersonic or subsonic,
as shown  in Fig. \ref{figure:shock refelction} and
Fig. \ref{figure:shock refelction-subs}, respectively.

\begin{definition}\label{def:weak solutionShockRefl}
$(\rho,\vv)\in L^{\infty}(\Lambda)$, with
$\rho\in BV_{\rm loc}(\Lambda\cap {\mathcal N}_r(\partial\Lambda))$ for some $r>0$,
is called an entropy solution of the regular
shock reflection problem if $(\rho, \vv)$ is an entropy
solution of system
\eqref{selfSimEuler}--\eqref{selfSimEuler-moment}
with slip boundary condition \eqref{Euler-slipBC-def}
in the sense of {\rm Definition \ref{weakSolEulerSlipBC}},
which satisfies
the asymptotic conditions{\rm :}
$$\displaystyle
\lim_{R\to\infty}\|(\rho, \vv)-(\bar{\rho},\bar{\vv})\|_{0, \Lambda\setminus
B_R(0)}=0,
$$
where
\begin{equation*}
(\bar{\rho},\bar{\vv})=
\begin{cases} (\rho_0, \vv_0) \qquad\mbox{for}\,\,\,
                         \xi_1>\xi_1^0,\\[1mm]
              (\rho_1, \vv_1) \qquad \mbox{for}\,\,\,
                          \xi_1<\xi_1^0,
\end{cases}
\end{equation*}
and $\xi_1^0>0$ is the location of the incident shock $S_0$ on the self-similar plane.
\end{definition}

Next, we define the points and lines in
Figs. \ref{figure:shock refelction}--\ref{figure:shock refelction-subs}.
The incident shock $S_0$ is line $\{\xi_1=\xi_1^0\}$ with
 $\xi_1^0=\frac{\rho_1{u}^{(1)}_1}{\rho_1-\rho_0}>0$.
The center,  $O_2=\uu^{(2)}=(u_1^{(2)}, u_2^{(2)})$, of the sonic circle
$B_{c_2}(O_2)$ of state $(2)$
lies on the wedge boundary between the reflection point $P_0$
and the wedge vertex $P_3$
for both the supersonic and subsonic cases.

Then,  for the supersonic case, {\it i.e.},
when $|{\mathbf v}_2(P_0)|=|P_0 O_2|>c_2$ so that
$P_0\notin \overline{B_{c_2}(O_2)}$,
we denote by
$P_4$ the {\it upper  point} of intersection of
$\partial B_{c_2}(O_2)$ with the wedge boundary
such that $O_2\in P_3P_4$.
Also, the sonic circle $\partial B_{c_2}(O_2)$ of state $(2)$ intersects
line $S_1$,
and one of the points of intersection, $P_1\in\Lambda$, is such that
segment $P_0P_1$ is outside $B_{c_2}(O_2)$.
Denote  the arc of $\partial B_{c_2}(O_2)$ by $\Gso=P_1P_4$.
The curved part of the reflected-diffracted shock is $\Gsh=P_1P_2$,
where $P_2\in\{\xi_2=0\}$.
Then we denote the line segments $\Gamma_{\rm sym}:=P_2P_3$ and $\Gw:=P_3P_4$.
The lines and curves $\Gsh$, $\Gso$, $\Gamma_{\rm sym}$, and $\Gw$ do not have
common points, except for their endpoints $P_1, \cdots, P_4$.
Thus, $\Gsh\cup\Gso\cup\Gamma_{\rm sym}\cup\Gw$ is a closed curve without
self-intersection.
Denote by $\Omega$ the bounded domain restricted by this curve.

For the subsonic/sonic case, {\it i.e.},
when $|{\mathbf v}_2(P_0)|=|P_0 O_2|\le c_2$ so that $P_0\in \overline{B_{c_2}(O_2)}$,
the curved reflected-diffracted shock
is $\Gsh=P_0P_2$, which does not have common interior points with
the line segments $\Gamma_{\rm sym}=P_2P_3$ and $\Gw=P_0P_3$.
Then $\Gsh\cup\Gamma_{\rm sym}\cup\Gw$ is a closed curve without
self-intersection, and $\Omega$ is the bounded domain restricted
by this curve.

Furthermore, in some parts of the argument below, it is convenient to
extend problem \eqref{selfSimEuler}--\eqref{selfSimEuler-moment} and (\ref{Euler-slipBC-def}),
given in $\Lambda$ by even reflection about the $\xi_1$--axis, {\it i.e.},
defining
$$
(\rho^{\rm ext}, \vv^{\rm ext})(-\xi_1, \xi_2)
:=(\rho^{\rm ext}, \vv^{\rm ext})(\xi_1, \xi_2)
\qquad\,\,\, \mbox{for any $\xxi=(\xi_1,\xi_2)\in\Lambda$}.
$$
Then $(\rho^{\rm ext}, \vv^{\rm ext})$ is defined in region
$\Lambda^{\rm ext}$ obtained from $\Lambda$ by adding the reflected region $\Lambda^-$,
{\it i.e.}, $\Lambda^{\rm ext}=\Lambda\cup\{(\xi_1, 0)\; :\;\xi_1<0\}\cup\Lambda^-$.
In a similar way, region $\Omega$ and curves
$\Gsh\subset\partial\Omega$ and $P_0P_2$ can be extended
into the corresponding
region $\Omega^{\rm ext}$ and curves
$\Gsh^{\rm ext}\subset\partial\Omega^{\rm ext}$ and $P_0P_2 P_0^{\rm ext}$.

\medskip
Now we give the definition of a global regular shock reflection solution.
The intuition for the definition is the following:
The regular shock reflection solution is an entropy solution of the regular shock reflection
problem in the sense of Definition \ref{def:weak solutionShockRefl} which has the structure shown
in Figs. \ref{figure:shock refelction}--\ref{figure:shock refelction-subs},
where $(\rho, \vv)$ coincides with those of states $(0)$, $(1)$, and $(2)$ in their respective regions.
As we discussed above, the necessary condition is the existence of
state $(2)$, which means that the wedge angle satisfies $\theta_{\rm w}\in (\theta_{\rm w}^{\rm d}, \frac{\pi}{2})$.
Moreover, it is expected that the solution is relatively regular in $\Omega$.
However, we show below that it is not possible that $\vv\in H^1(\Omega)$.
On the other hand, from the physical/computational experiments and the theoretical results
in the case of potential flow,
it is expected that $\Shock$ is a smooth curve and $\vv$ is smooth near and up to $\Shock\cup \Sonic$ in $\Omega$.
Moreover, the regularity discussed is expected for the shock reflection-diffraction configuration extended to $\{\xi_1<0\}$
by the even reflection about the $\xi_1$--axis (since this is the original shock reflection-diffraction configuration).
In particular, the extended shock curve $\Gsh^{\rm ext}$ is smooth, which shows
that $\Shock$ must be orthogonal to $\Symm$ at $P_2$.
Then, noting that the pseudo-velocity of state $(1)$ is
$\vv_1(\xxi)=(u_1^{(1)}, 0)-\xxi$ with $u_1^{(1)}>0$ and $\{P_2\}=\overline\Shock\cap \overline\Symm\subset\{\xi_1<0\}$,
it follows that vector $\vv_1$ on $\Shock$ near $P_2$ points into $\Omega$.
It is then expected that this holds on the whole shock $\Shock$, unless the jumps of the velocity and the density across
the reflected-diffracted shock are degenerate at some points (which are not expected).
Then it follows from \eqref{shock-RHp} for $\Shock$
that $\vv$ on $\Shock$ from the $\Omega$--side also points into $\Omega$.
On the sonic arc (for the supersonic reflection),
the jump of the velocity is not expected, and velocity $\vv_2$ of state $(2)$ on $\Sonic$ points into $\Omega$ (as $\vv_2$ points
along the radial direction of the sonic circle of state $(2)$ towards its center $O_2\in P_3P_4$).
Therefore, it is expected that
\begin{equation}\label{velocityDirectionOnShock}
  \vv\cdot \nnu \le -C^{-1}\qquad
  \mbox{ on $\overline\Shock\cup \overline\Sonic$}
\end{equation}
for some $C>0$,
where $\vv$ on the curves is taken from the $\Omega$--side and $\nnu$ is the outer normal with respect to $\Omega$.

Based on the remarks above, we define the notion of regular shock reflection
solutions:

\begin{definition}\label{regReflSolDef}
Fix the wedge angle $\theta_{\rm w}\in (\theta_{\rm w}^{\rm d}, \frac{\pi}{2})$, and let
domain $\Lambda=\Lambda(\theta_{\rm w})$ as defined above.
An entropy solution
$(\rho,\vv)$ of the
regular shock reflection
problem in the sense of
{\rm Definition \ref{def:weak solutionShockRefl}} is called a
regular shock reflection solution if $(\rho,\vv)$ satisfies
the following additional properties{\rm :}

If state $(2)$ for $\theta_{\rm w}$ is supersonic at $P_0$, i.e., $|\vv_2(P_0)|>c_2$,
the  solution has the supersonic reflection structure as on {\rm Fig. \ref{figure:shock refelction}}.
If state $(2)$ for $\theta_{\rm w}$ is subsonic or sonic at $P_0$, i.e., $|\vv_2(P_0)|\le c_2$,
the  solution has the subsonic reflection structure as on {\rm Fig. \ref{figure:shock refelction-subs}}.
More specifically,
\begin{enumerate}[\rm (i)]
\item \label{regReflSolDef-i1}
The extended reflected-diffracted shock curve $P_0P_2P_0^{\rm ext}$ is $C^1$ up to its endpoints.

\item \label{regReflSolDef-i2}
$(\rho, \vv)$ is continuous in $\overline\Omega\cap {\mathcal N}_r(\Shock\cup\Sonic)$
for the supersonic reflection, and in $\overline\Omega\cap {\mathcal N}_r(\Shock)$ for the
subsonic or sonic reflection, for some $r>0$.

\item \label{regReflSolDef-i3}
The solution coincides with states $(0)$, $(1)$, and $(2)$ in their respective regions{\rm :}
for the supersonic reflection case,
\begin{equation*}
(\rho, \vv)=
\begin{cases} (\rho_0, \vv_0) \qquad\mbox{for $\xi_1>\xi_1^0$ and $\xi_2>\xi_1 \tan\theta_{\rm w}$},\\[1mm]
              (\rho_1, \vv_1) \qquad \mbox{for $\xi_1<\xi_1^0$ and above curve $P_0P_1P_2$}, \\[1mm]
              (\rho_2, \vv_2) \qquad \mbox{in $P_0P_1P_4$},
\end{cases}
\end{equation*}
where $\xi_1^0>0$ is the location of the incident shock $S_0$ on the self-similar plane{\rm ;}
and for the subsonic or sonic reflection case,
\begin{equation*}
(\rho, \vv)=
\begin{cases} (\rho_0, \vv_0) \qquad\mbox{for $\xi_1>\xi_1^0$ and $\xi_2>\xi_1 \tan\theta_{\rm w}$},\\[1mm]
              (\rho_1, \vv_1) \qquad \mbox{for $\xi_1<\xi_1^0$ and above curve $P_0P_2$},
\end{cases}
\end{equation*}
and $\displaystyle \lim_{\xxi\in\Omega,\;\xxi\to P_0}(\rho, \vv)(\xxi)=(\rho_2, \vv_2)(P_0)$.

\item  \label{regReflSolDef-i4}
\eqref{velocityDirectionOnShock} holds  for some $C>0$, where $\vv$ on the curves is taken from the $\Omega$--side
and $\nnu$ is the outer normal with respect to $\Omega$.

\item  \label{regReflSolDef-i5}
The flow is pseudo-subsonic in $\Omega$ on and near $\Shock$, except for the sonic point $P_1$ for the supersonic
or sonic reflections, i.e.,
$|\vv|<c$ on $\overline\Shock\setminus\{P_1\}$ for supersonic and sonic reflections, and on
$\overline\Shock$ for strictly subsonic reflections, where $\vv$ on $\Shock$ is computed
from the $\Omega$--side.
\end{enumerate}
\end{definition}

\begin{remark}\label{shockNotStraight}
The curve, $\Shock$, cannot be a straight segment. Indeed, if the shock is a straight segment,
then it lies on a vertical line passing through $P_2$, since the tangent line to $\Shock$ at $P_2$ is vertical
by condition \eqref{regReflSolDef-i1} of {\rm Definition \ref{regReflSolDef}}.
On the other hand, the tangent to $\Shock$ at $P_1$ for the supersonic reflection and at $P_0$ for the subsonic
reflection is tangent to the straight shock $S_1$ between state $(1)$ and $(2)$,
where, for the subsonic case,
this follows from the property{\rm :}
$\displaystyle \lim_{\xxi\in\Omega,\;\xxi\to P_0}(\rho, \vv)(\xxi)=(\rho_2, \vv_2)(P_0)$
in  {\rm Definition \ref{regReflSolDef}\eqref{regReflSolDef-i3}} by
using the Rankine-Hugoniot conditions  \eqref{RH}--\eqref{shock-RHp} on the shock.
As we have shown above, the straight shock between states $(1)$ and $(2)$ is not vertical. This shows
that $\Shock$ cannot lie in a straight line.
\end{remark}

\begin{remark}\label{shockNotStraight-b}
For the supersonic reflections, it follows from conditions {\rm (ii)--(iii)} of  {\rm Definition \ref{regReflSolDef}} 
and the Rankine-Hugoniot conditions
that $(\rho, \vv)$ is continuous across the sonic arc $P_1P_4=\overline{\Sonic}${\rm :}
\begin{equation}\label{4.1a}
 \lim_{\xxi\in\Omega,\;\xxi\to P}(\rho, \vv)(\xxi)=(\rho_2, \vv_2(P))\qquad\mbox{for any $P\in\overline{\Sonic}$}.
\end{equation}
\end{remark}

\begin{theorem}\label{lower-reg-Th-RegRefl-a}
Let $(\rho,\vv)$ be a regular shock reflection solution in the sense of {\rm Definition \ref{regReflSolDef}}.
Assume that $(\rho,\vv)$ satisfies the following{\rm :}
\begin{enumerate}[\rm (i)]
\item\label{lower-reg-Th-RegRefl-i1-a}
The reflected-diffracted shock $\Shock=P_1P_2$ is $C^2$ in its relative interior, and $C^{1}$ up to endpoints $P_1$, where $P_1$
is replaced by $P_0$ if the subsonic shock reflection occurs at $P_0${\rm ;}
\item\label{lower-reg-Th-RegRefl-i2-a}
$(\rho, \vv)\in C^1\left((\Nbhd_\sigma(\Shock\cup\Sonic)\cap\overline\Omega)\setminus\{P_1\}\right)\cap C^{0,1}(\Nbhd_\sigma(\Shock\cup\Sonic)\cap\overline\Omega)$
for some $\sigma>0$, where $P_1$ is replaced by $P_0$ if the subsonic shock reflection occurs at $P_0${\rm ;}
\item\label{lower-reg-Th-RegRefl-i3-a}
$|\vv|\le C_0 $ and $C_0^{-1}\le \rho\le C_0$ in $\Omega$ for some $C_0\ge 1$.
\end{enumerate}
Then $\vv\notin H^1(\Omega)$.
\end{theorem}

The proof of Theorem \ref{lower-reg-Th-RegRefl-a} will be given for the more general non-symmetric case below (Theorem \ref{lower-reg-Th-RegRefl}),
based on the observation in Remark \ref{4.7a}.

\bigskip
%%%%%%%%%%%%%%%%%%%%%%%%%%%%%%%%%%%%%%%%%%%%%%%%%%%%%%%%%%%%%%%%%%%%%%%%%%%%%%%%%%%%%%%%%%%%%
Next, for the non-symmetric regular shock reflection problem,
based on the argument near the reflection point $P_0$ for the symmetric case above,
there are four configurations depending on whether the solution at $P_0$ and $P_1$ is subsonic or supersonic.
As two examples of them, Fig. \ref{figure:shock refelction nonsym} and Fig. \ref{figure:shock refelction-subs nonsym}
exhibit the structures of two supersonic regular reflections and two subsonic regular reflections, respectively.
In light of Definition \ref{regReflSolDef}, we define the regular reflection solution for the non-symmetric case.
Let the wedge angle $\theta_{\mathrm{w}}=\theta_{\mathrm{w}}^1+\theta_{\mathrm{w}}^2$
with $\theta_{\mathrm{w}}^i\in(\theta_{\mathrm{w}}^{\mathrm{d}},\frac{\pi}{2})$,
where $\theta_{\mathrm{w}}^i$ is the angle between $\Gamma_{\mathrm{w}}^i$ and the $\xi_1$--axis for $i=1,2$. Let
$$
\Lambda=\bR^2_+\setminus\big\{\boldsymbol{\xi}\; : \; \xi_ 1>0,\; -\xi_1\tan\theta_{\rm w}^2<\xi_2<\xi_1\tan\theta_{\rm w}^1\big\}.
$$
Let $\xi_1^0>0$ be the location of the incident shock $S_0$ on the self-similar plane.

\begin{figure}[htp]
\begin{center}
	\begin{minipage}{0.52\textwidth}
		\centering
		\includegraphics[width=0.6\textwidth]{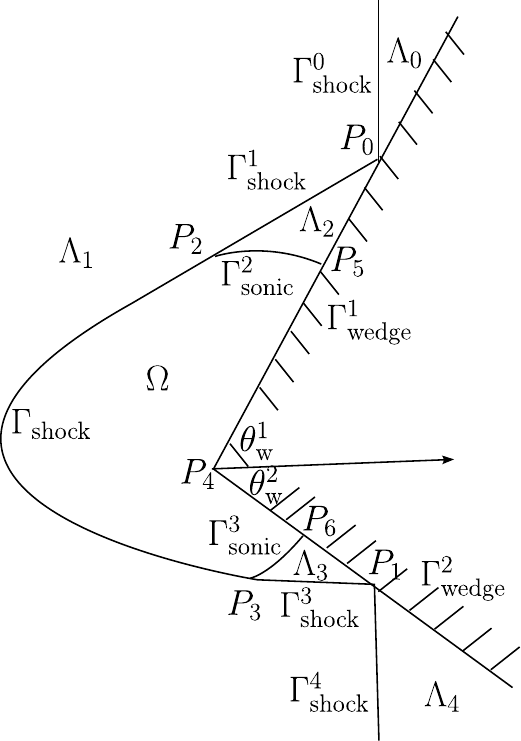}
\caption{Non-symmetric supersonic regular reflection}
\label{figure:shock refelction nonsym}
		\end{minipage}
	\hspace{-0.55in}
	\begin{minipage}{0.53\textwidth}
		\centering
		\includegraphics[width=0.6\textwidth]{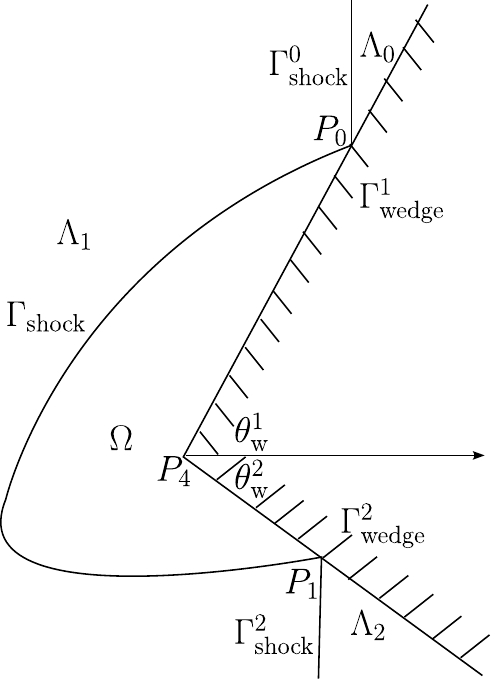}
\caption{Non-symmetric subsonic regular reflection}
\label{figure:shock refelction-subs nonsym}
			\end{minipage}
\end{center}
\end{figure}

We call $(\rho,\vv)\in L^{\infty}(\Lambda)$, with
$\rho\in BV_{\rm loc}(\Lambda\cap {\mathcal N}_r(\partial\Lambda))$ for some $r>0$,
is an entropy solution of the non-symmetric regular shock reflection problem
if $(\rho, \vv)$ is an entropy
solution of system
\eqref{selfSimEuler}--\eqref{selfSimEuler-moment}
with the slip boundary condition \eqref{Euler-slipBC-def}
in the sense of Definition \ref{weakSolEulerSlipBC}, which satisfies
the asymptotic conditions{\rm :}
$$
\displaystyle
\lim_{R\to\infty}\|(\rho, \vv)-(\bar{\rho},\bar{\vv})\|_{0, \Lambda\setminus
B_R(0)}=0,
$$
where
\begin{equation*}
(\bar{\rho},\bar{\vv})=
\begin{cases} (\rho_0, \vv_0) \qquad\mbox{for $\xi_1>\xi_1^0$},\\[1mm]
             (\rho_1, \vv_1) \qquad \mbox{for $\xi_1<\xi_1^0$}.
\end{cases}
\end{equation*}
Then we can define the non-symmetric regular shock reflection solution.

\begin{definition}\label{def-shock-reflection-nonsym}
For given angles $\theta_{\mathrm{w}}^i$ with $\theta_{\mathrm{w}}^i\in(\theta_{\mathrm{w}}^{\mathrm{d}},\frac{\pi}{2})$ for $i=1,2$,
an entropy solution $(\rho,\mathbf{v})$ is called a non-symmetric regular shock reflection solution of
the non-symmetric regular shock reflection problem
if  $(\rho,\mathbf{v})$ satisfies the following additional properties{\rm :}
\begin{enumerate}[\rm (i)]
\item \label{nonregReflSolDef-i0} If state $(2)$ at $P_0$ and state $(3)$ at $P_1$ are both supersonic, i.e.,
$|\vv_2(P_0)|>c_2$ and $|\vv_2(P_1)|>c_3$, the solution has the supersonic reflection structure
as in {\rm Fig. \ref{figure:shock refelction nonsym}} at both points $P_0$ and $P_1$.
If
state $(2)$ at $P_0$ and state $(3)$ at $P_1$ are both subsonic or sonic, i.e.,
$|\vv_2(P_0)|\le c_2$ and  $|\vv_2(P_1)|\le c_3$,
the solution has the subsonic reflection structure as in {\rm Fig. \ref{figure:shock refelction-subs nonsym}}.

\smallskip
\item \label{nonregReflSolDef-i1} The reflected-diffracted shock curve $P_0P_1$ is $C^1$ up to its endpoints.

\smallskip
\item \label{nonregReflSolDef-i2}
$(\rho, \vv)$ is continuous in $\overline\Omega\cap {\mathcal N}_r(\Shock\cup\Sonic^2\cup\Sonic^3)$ for the supersonic reflections,
and in $\overline\Omega\cap {\mathcal N}_r(\Shock)$ for the subsonic reflection, for some $r>0$.

\smallskip
\item \label{nonregReflSolDef-i3} The solution coincides with states $(0)$, $(1)$, $(2)$, and $(3)$ in their respective regions{\rm :}
\begin{equation*}
(\rho, \vv)=
\begin{cases} (\rho_0, \vv_0) \qquad\mbox{for}\,\,\,
                         \xi_1>\xi_1^0,\,\,\, \xi_2>\xi_1 \tan\theta_{\rm w}^1,\\[1mm]
            (\rho_0, \vv_0) \qquad\mbox{for}\,\,\,
                         \xi_1>\xi_1^0,\,\,\, \xi_2<-\xi_1 \tan\theta_{\rm w}^2,\\[1mm]
              (\rho_1, \vv_1) \qquad \mbox{for}\,\,\,
                          \xi_1<\xi_1^0, \,\,\,
                          \mbox{and the left to curve }\;P_0P_1, \\[1mm]
              (\rho_2, \vv_2) \qquad \mbox{in}\,\,\,\Lambda_2\,\,\, \mbox{if state $(2)$ at $P_0$ is supersonic},\\[1mm]
         (\rho_3, \vv_3) \qquad \mbox{in}\,\,\,\Lambda_3\,\,\, \mbox{if state $(3)$ at $P_1$ is supersonic},
\end{cases}
\end{equation*}
where $\xi_1^0>0$ is the location of the incident shock $S_0$ on the self-similar plane, and
$$
\displaystyle \lim_{\xxi\in\Omega,\;\xxi\to P_0}(\rho, \vv)(\xxi)=(\rho_2, \vv_2)(P_0)\qquad\mbox{if state $(2)$ at $P_0$ is subsonic or sonic},
$$
$$
\displaystyle \lim_{\xxi\in\Omega,\;\xxi\to P_1}(\rho, \vv)(\xxi)=(\rho_3, \vv_3)(P_1)\qquad\mbox{if state $(3)$ at $P_1$ is subsonic or sonic}.
$$

\item \label{nonregReflSolDef-i4}
\eqref{velocityDirectionOnShock} holds for some $C>0$, where $\Sonic$ is replaced by $\Sonic^2\cup\Sonic^3$,
$\vv$ on the curves is taken from the $\Omega$--side, and $\nnu$ is the outer normal with respect to $\Omega$.

\item \label{nonregReflSolDef-i5}
The flow is pseudo-subsonic in $\Omega$ on and near  $\Shock$, except for the sonic points $P_2$ and $P_3$ if they exist,
point $P_0$ if state $(2)$ is sonic at $P_0$,
and point $P_1$ if state $(3)$ is sonic at $P_1$.
\end{enumerate}
\end{definition}

\begin{remark}\label{4.7a}
Similarly as stated in {\rm Remark \ref{shockNotStraight}}, curve $\Shock$ cannot be a straight segment, which will be
addressed in more detail in {\rm Step 3} of the proof of {\rm Theorem \ref {lower-reg-Th-RegRefl}} below.
Moreover, if $\theta^1_{\mathrm{w}}=\theta^2_{\mathrm{w}}$, the symmetric regular shock reflection problem is a special case of
the non-symmetric regular shock reflection problem,
by defining $(\rho^{\rm ext}, \vv^{\rm ext})(\xi_1, -\xi_2)):=(\rho^{\rm ext}, \vv^{\rm ext})(\xi_1, \xi_2)$ for any $\xxi=(\xi_1,\xi_2)\in\Lambda$.
Therefore, we only prove the low regularity of the non-symmetric regular reflection solution below,
since the low regularity for the symmetric case follows directly.
\end{remark}

\begin{remark}\label{4.8a}
For the supersonic reflections, it follows from conditions {\rm (ii)--(iv)} of {\rm Definition \ref{def-shock-reflection-nonsym}}
and the Rankine-Hugoniot conditions that $(\rho, \vv)$ is continuous across the sonic arcs $\Sonic^2\cup\Sonic^3${\rm :}
$$
 \lim_{\xxi\in\Omega,\;\xxi\to P}(\rho, \vv)(\xxi)=(\rho_2, \vv_i)(P)\qquad\mbox{for any }\;P\in\overline{\Sonic^i}\quad\mbox{if }\;
 |\vv_i(P_0)|>c_i\;\; \mbox{for }\; i=2,3.
$$
\end{remark}

%%%%%%%%%%%%%%%%%%%%%%%%%%%%%%%%%%%%%%%%%%%%%%%%%%%%%%%%%%%%%%%%%%%%%%%%%%%%%%%%%%%%%%%%%

\begin{theorem}\label{lower-reg-Th-RegRefl}
Let $(\rho,\vv)$ be a non-symmetric regular shock reflection solution in the sense of {\rm Definition \ref{def-shock-reflection-nonsym}}.
Assume that $(\rho,\vv)$ satisfies the following{\rm :}
\begin{enumerate}[\rm (i)]
\item\label{lower-reg-Th-RegRefl-i1}
The reflected-diffracted shock $\Shock=P_2P_3$ is $C^2$ in its relative interior, and $C^{1}$ up to endpoints $P_2$ and $P_3$, where $P_2$ {\rm (}or $P_3${\rm )}
is replaced by $P_0$ {\rm (}or $P_1${\rm )} if the subsonic shock reflection occurs at $P_0$ {\rm (}or $P_1${\rm )}{\rm ;}
\item\label{lower-reg-Th-RegRefl-i2}
{\small $(\rho, \vv)\in C^1\left((\Nbhd_\sigma(\Shock\cup\Sonic^1\cup\Sonic^2)\cap\overline\Omega)\setminus\{P_2,\,P_3\}\right)\cap C^{0,1}(\Nbhd_\sigma(\Shock\cup\Sonic^1\cup\Sonic^2)\cap\overline\Omega)$}
for some $\sigma>0$, where $P_2$ {\rm (}or $P_3${\rm )} is replaced by $P_0$ {\rm (}or $P_1${\rm )} if the subsonic shock reflection occurs at $P_0$ {\rm (}or $P_1${\rm )}{\rm ;}
\item\label{lower-reg-Th-RegRefl-i3}
$|\vv|\le C_0 $ and $C_0^{-1}\le \rho\le C_0$ in $\Omega$ for some $C_0\ge 1$.
\end{enumerate}
Then $\vv\notin H^1(\Omega)$.
\end{theorem}
\begin{proof}
It suffices to check that the assumptions of Theorem \ref{lower-reg-Th} are satisfied.
We divide the proof into four steps:

\smallskip
1. Since the cases with one supersonic reflection and one subsonic reflection can be treated similarly,
we focus only on the non-symmetric supersonic reflection solution (see Fig. \ref{figure:shock refelction nonsym})
and the non-symmetric subsonic reflection solution (see Fig. \ref{figure:shock refelction-subs nonsym}) in the proof below.

\smallskip
2. We first show that the non-symmetric regular reflection solutions satisfy
conditions  \eqref{RiemannProblSolutStruct-1-i1}--\eqref{RiemannProblSolutStruct-1-i6} of  Definition \ref{RiemannProblSolutStruct-1}
and assumptions \eqref{lower-reg-Th-i2}--\eqref{lower-reg-Th-i4} of Theorem \ref{lower-reg-Th}.

The properties described in
Definition \ref{RiemannProblSolutStruct-1}\eqref{RiemannProblSolutStruct-1-i1} with $M=5$
for the non-symmetric supersonic reflection solutions
and with $M=3$ for the non-symmetric subsonic reflection solutions
by Definition \ref{def-shock-reflection-nonsym}\eqref{nonregReflSolDef-i3},
and with $\Lambda_j$ being the region of possible
state $(0)$, $(1)$, $(2)$, or $(3)$.
Moreover, condition \eqref{RiemannProblSolutStruct-1-i2} of Definition \ref{RiemannProblSolutStruct-1} also holds
for the regular shock reflection solutions because states $(0)$, $(1)$, $(2)$, and $(3)$ are four different constant states,
specifically $\rho_0<\rho_1<\rho_2$ and $\rho_0<\rho_1<\rho_3$.

The properties in  Definition \ref{RiemannProblSolutStruct-1}\eqref{RiemannProblSolutStruct-1-i3} hold,
{\it i.e.}, $\partial\Omega$ is a Lipschitz curve,
by Definition \ref{def-shock-reflection-nonsym}\eqref{nonregReflSolDef-i1}
and the fact that the rest of $\partial\Omega$
consists of the straight segments and the arcs of the circles for the supersonic reflection,
and the angles of vertices $P_j, j=1,2,\cdots, 6$, for the non-symmetric supersonic reflection
and of vertices $P_0$, $P_1$, and $P_4$ for the non-symmetric subsonic reflection are all within $(0, \pi)$.

For the non-symmetric supersonic reflection, $\Nex=2$ and $\Nint=3$, with $\GammaExt_1=\Wedge^1=P_4P_5$,
  $\GammaExt_2=\Wedge^2=P_4P_6$,  $\GammaInt_1=\Shock=P_2P_3$,
  $\GammaInt_2=\Sonic^2=P_2P_5$, and $\GammaInt_3=\Sonic^3=P_3P_6$.

For the non-symmetric subsonic reflection, $\Nex=2$ and $\Nint=1$, with $\GammaExt_1=\Wedge^1=P_0P_4$,
  $\GammaExt_2=\Wedge^2=P_1P_4$, and  $\GammaInt_1=\Shock=P_0P_1$.

In both the non-symmetric supersonic and subsonic reflection cases, all the requirements of
Definition \ref{RiemannProblSolutStruct-1}\eqref{RiemannProblSolutStruct-1-i4}--\eqref{RiemannProblSolutStruct-1-i5} hold,
by the regularity of $\Shock$ given in  Definition \ref{def-shock-reflection-nonsym}\eqref{nonregReflSolDef-i1}
and the facts that the angles in the corner points of $\Omega$
are within $(0, \pi)$ for the non-symmetric regular reflection except for point $P_4$,
where condition \eqref{domainLambda} holds with $\theta^-=\thetaw^1$ and $\theta^+=2\pi-\thetaw^2$.
Also, point $P_4=(0,0)$ for the non-symmetric regular reflection solution is point
$\PZer=(0, 0)$ as described in Definition \ref{RiemannProblSolutStruct-1}\eqref{RiemannProblSolutStruct-1-i4}.
Moreover, $P_4$ is the common point of
$\GammaExt_1=\Wedge^1$ and
$\GammaExt_2=\Wedge^2$ for both the regular supersonic and subsonic reflections,
as described in Definition \ref{RiemannProblSolutStruct-1}\eqref{RiemannProblSolutStruct-1-i4}.

Now we show that the requirements of
Definition \ref{RiemannProblSolutStruct-1}\eqref{RiemannProblSolutStruct-1-i6} hold.
Clearly, $\GammaInt_1=\Shock$ is a shock,
and $\vv \cdot \nnu\le -C^{-1}$
where $\vv$ on $\GammaInt_{1}$ is taken from the $\Omega$--side and
$\nnu$ is the outer normal with respect to $\Omega$ by property \eqref{nonregReflSolDef-i4} of
Definition \ref{def-shock-reflection-nonsym} of the non-symmetric regular reflection solution.
Also, for the non-symmetric supersonic reflection,
$\GammaInt_2=\Sonic^2$ and $\GammaInt_3=\Sonic^3$ are an arc of the sonic circles of state $(2)$ and state $(3)$, respectively,
which imply that
$\vv_2\cdot\nnu=-|\vv_2|=-c_2<0$ on $\Sonic^2$ and $\vv_3\cdot\nnu=-|\vv_3|=-c_3<0$ on $\Sonic^3$.
Using the boundary conditions that $\vv=\vv_2$ on $\Sonic^2$ and $\vv=\vv_3$ on $\Sonic^3$ by
Definition \ref{def-shock-reflection-nonsym}\eqref{nonregReflSolDef-i3},
we obtain that
$$
\vv\cdot\nnu=-c_2<0\,\,\,\, \mbox{ on $\Sonic^2$}, \qquad\,\,\,\, \vv\cdot\nnu=-c_3<0\,\,\,\, \mbox{ on $\Sonic^3$}.
$$
Thus, the requirements of   Definition \ref{RiemannProblSolutStruct-1}\eqref{RiemannProblSolutStruct-1-i6} hold.

Next, since $\overline{\GammaInt}=\overline{\Shock}\cup\overline{\Sonic^1}\cup\overline{\Sonic^2}$,
assumptions \eqref{lower-reg-Th-RegRefl-i2}--\eqref{lower-reg-Th-RegRefl-i3}
of Theorem \ref{lower-reg-Th-RegRefl} imply that
assumptions \eqref{lower-reg-Th-i2}--\eqref{lower-reg-Th-i3} of Theorem \ref{lower-reg-Th} hold
for the regular shock reflection solution.

Therefore, we have shown that conditions  \eqref{RiemannProblSolutStruct-1-i1}--\eqref{RiemannProblSolutStruct-1-i6}
of Definition \ref{RiemannProblSolutStruct-1} and assumptions \eqref{lower-reg-Th-i2}--\eqref{lower-reg-Th-i3}
of Theorem \ref{lower-reg-Th} hold.

\smallskip
3. It remains to show that the requirements in Definition \ref{RiemannProblSolutStruct-1}\eqref{RiemannProblSolutStruct-1-i7} hold.
This is achieved by the use of Lemma \ref{PtWithNonzeroTangVelocCurv-lemma}.
Thus, it remains to show that conditions \eqref{PtWithNonzeroTangVelocCurv-lemma-i1}--\eqref{PtWithNonzeroTangVelocCurv-lemma-i3}
of Lemma \ref{PtWithNonzeroTangVelocCurv-lemma} are satisfied below.

Condition \eqref{PtWithNonzeroTangVelocCurv-lemma-i1} is satisfied since $\Lambda_1$ is
the region of state $(1)$ as shown in Figs. \ref{figure:shock refelction nonsym}--\ref{figure:shock refelction-subs nonsym} above,
and $\Shock=\partial\Lambda_1\cap\partial\Omega$ in both the supersonic and subsonic cases.

We now check condition \eqref{PtWithNonzeroTangVelocCurv-lemma-i2}  of Lemma \ref{PtWithNonzeroTangVelocCurv-lemma}.
Assume that $\Shock$ is a straight segment. Recall that $S_1=\Shock^1$ is the straight shock between
states $(1)$ and $(2)$, {\it i.e.}, the line passing through points $P_0$ and $P_2$
for the supersonic reflection, and through $P_0$ for the subsonic reflection.
Using  Definition \ref{def-shock-reflection-nonsym}\eqref{nonregReflSolDef-i3}, we obtain that
  $(\rho, \vv)(P_2)=(\rho_2,\vv_2)(P_2)$  for the non-symmetric supersonic reflection, and
  $(\rho, \vv)(P_0)=(\rho_2,\vv_2)(P_0)$  for the non-symmetric subsonic reflection,
where $(\rho, \vv)$ in both cases is computed from the $\Omega$--side.
Thus, the tangent line to $\GammaInt_1=\Shock$ at the upper endpoint, {\it i.e.}, at $P_2$
for the non-symmetric supersonic reflection and at $P_0$ for
the non-symmetric subsonic reflection, is line $S_1$.
Similarly, we have the tangent line to $\GammaInt_1=\Shock$ at the lower endpoint, {\it i.e.},
at $P_3$ for the non-symmetric supersonic reflection and at $P_1$ for the non-symmetric subsonic reflection,
is line $S_2=\Shock^2$.
Hence, if $\Shock$ is a straight segment, then $\Shock$ lies on both lines $S_1$ and $S_2$;
in particular, these lines coincide.
It follows that $S_1=S_2$ is the line passing through $P_0$ and $P_1$,
and $\Shock$ lies within interval $P_0P_1$.
However, the wedge is convex since   $\theta_{\mathrm{w}}^i \in (0, \frac\pi 2)$ for $i=1,2$,
so $P_0P_1$ lies within the wedge, thus outside $\Lambda$.
It follows that $\Shock\subset P_0P_1$ lies outside $\Lambda$, which contradicts the structure
of regular reflection-diffraction configuration. That is,
the assumption that $\Shock$ is a straight segment leads to a contradiction, which verifies
\eqref{PtWithNonzeroTangVelocCurv-lemma-i2}  of Lemma \ref{PtWithNonzeroTangVelocCurv-lemma}.

Next, we show condition \eqref{PtWithNonzeroTangVelocCurv-lemma-i3} of Lemma \ref{PtWithNonzeroTangVelocCurv-lemma} is satisfied.

Consider first the case of reflections that are supersonic at $P_0$, {\it i.e.}, $|\vv_2(P_0)|>c_2$;
see Fig. \ref{figure:shock refelction nonsym}.
Then we need to show that $(\vv\cdot\ttau_{\rm shock})(P_2)\ne 0$. Assume this is not true, then
$(\vv\cdot\ttau_{\rm shock})(P_2)= 0$. Note that $\ttau_{\rm shock}(P_2)=\ttau_{S_1}$
and $\vv_{|\Omega}(P_2)=\vv_2(P_2)$
by Definition \ref{def-shock-reflection-nonsym}(\ref{nonregReflSolDef-i1}, \ref{nonregReflSolDef-i3})
so that $\vv_2(P_2)\cdot\ttau_{S_1}=0$ and
$$
|\vv(P_2)\cdot\nnu_{\rm shock}|=|\vv_2(P_2)\cdot\nnu_{S_1}|=|\vv_2(P_2)|=c_2,
$$
which contradicts the last inequality in \eqref{shock-entropy-RHp}. Thus
condition \eqref{PtWithNonzeroTangVelocCurv-lemma-i3} of Lemma \ref{PtWithNonzeroTangVelocCurv-lemma} is proved in the case when $|\vv_2(P_0)|>c_2$.

In the case of reflections that are sonic at $P_0$, {\it i.e.}, $|\vv_2(P_0)|=c_2$ (see Fig. \ref{figure:shock refelction-subs nonsym}),
the argument is the same as above with only notational change:  we use point $P_0$ here instead of point $P_2$.

It remains to consider the case of reflections that are subsonic at $P_0$, {\it i.e.}, $|\vv_2(P_0)|<c_2$;
see Fig. \ref{figure:shock refelction-subs nonsym}.
Assume that $(\vv\cdot\bt_{\rm shock})(P_0)=0$.
Recall that $(\rho, \vv)(P_0)=(\rho_2, \vv_2)(P_0)$  for the subsonic reflection
and $S_1$ is tangent to $\Shock$ at $P_0$, as we have shown above.
Thus, from $(\vv\cdot\bt_{\rm shock})(P_0)=0$, we obtain that
$(\vv_2\cdot\bt_{S_1)}(P_0)=0$.
Since $S_1$ is the line shock between states $(1)$ and $(2)$,
the last equality implies by \eqref{shock-RHp} that $(\vv_1\cdot\bt_{S_1)}(P_0)=0$, so that
line $L$ through centers $O_1=(u_1^{(1)}, 0)$ and $O_2=(u_1^{(2)}, u_2^{(2)})$ of states $(1)$ and $(2)$
(which is orthogonal to $S_1$) intersects $S_1$ at $P_0$.
Since $P_0\in\Wedge$, $O_2\in L\cap \Wedge$
(where $O_2\in\Wedge$ because $\vv_2\cdot\nnu=0$ on $\Wedge$),
and $L\cap \Wedge=\{P_0\}$, it follows
that $O_2=P_0$ so that $\vv_2(P_0)=\bf{0}$.
Since $\vv_2(P_0)=\bf{0}$,
it follows from the Rankine-Hugoniot conditions \eqref{RH}--\eqref{shock-RHp} on $S_1$ between states $(1)$ and $(2)$
that $\vv_1(P_0)=\bf{0}$, that is, $O_1=P_0$.
However, this is not true since $O_1=(u_1^{(1)}, 0)$ for $u_1^{(1)}>0$, while
$P_0=(l\cos\theta_{\rm w}^1, l\sin\theta_{\rm w}^1)$ for some $l>0$. This contradiction shows that
$(\vv\cdot\bt_{\rm shock})(P_0)\ne 0$, {\it i.e.},
condition \eqref{PtWithNonzeroTangVelocCurv-lemma-i3}
of Lemma \ref{PtWithNonzeroTangVelocCurv-lemma} holds for the subsonic reflection at $P_0$.

Now all the conditions of Lemma \ref{PtWithNonzeroTangVelocCurv-lemma} are verified for
the regular reflection solutions. Applying Lemma \ref{PtWithNonzeroTangVelocCurv-lemma},
we obtain that  Definition \ref{RiemannProblSolutStruct-1}\eqref{RiemannProblSolutStruct-1-i7} holds.

\smallskip
4. Now, all the conditions of Theorem \ref{lower-reg-Th} are verified for the regular reflection solutions.
Then the conclusion of Theorem \ref{lower-reg-Th-RegRefl} follows from
 Theorem \ref{lower-reg-Th}.
\end{proof}

%%%%%%%%%%%%%%%%%%%%%%%%%%%%%%%%%%%%%%%%%%%%%%%%
\subsection{Lower Regularity of the Prantl-Meyer Reflection Solutions for Supersonic Flows past a Solid Ramp}

The second example is the Prandtl reflection problem for the isentropic Euler system \eqref{isentropisEulersystem}.
This is of a self-similar structure that occurs when a $2$-D supersonic flow with density $\rho_{\infty}>0$ and velocity $\vv_{\infty}=(u_{\infty},0)$, $u_{\infty}>0$,
along the wedge-axis hits the wedge in the direction at $t=0$. See Figs. \ref{figure:Prandtl sup}--\ref{figure:Prandtl sub};
also see Bae-Chen-Feldman \cite{BCFQAM2013,BCF-2} and Elling-Liu \cite{ELCPAM2008}.

\begin{figure}[htp]
\begin{center}
	\begin{minipage}{0.52\textwidth}
		\centering
		\includegraphics[width=0.7\textwidth]{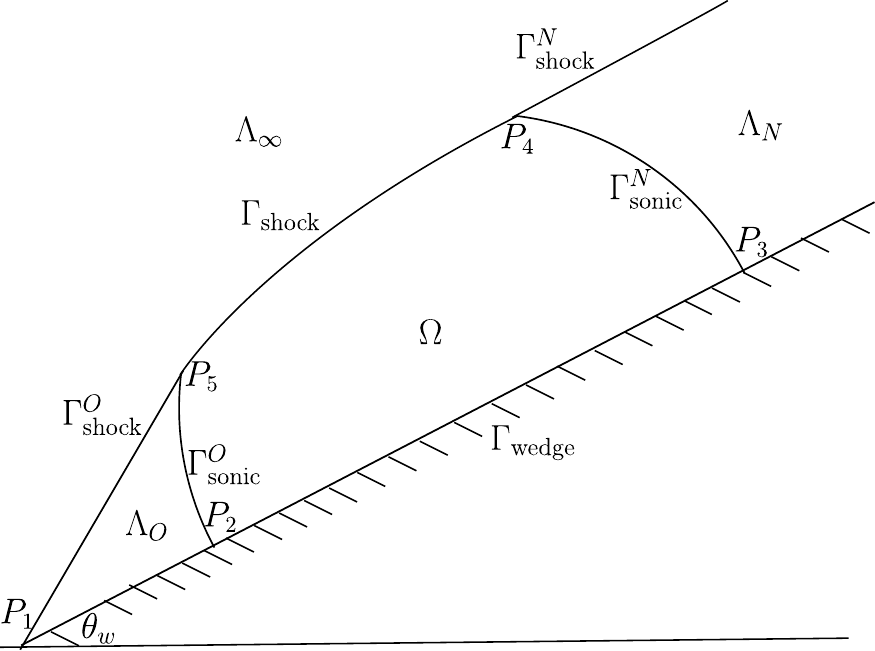}
\caption{Prandtl supersonic shock reflection}
\label{figure:Prandtl sup}
		\end{minipage}
	\hspace{-0.5in}
	\begin{minipage}{0.51\textwidth}
		\centering
		\includegraphics[width=0.7\textwidth]{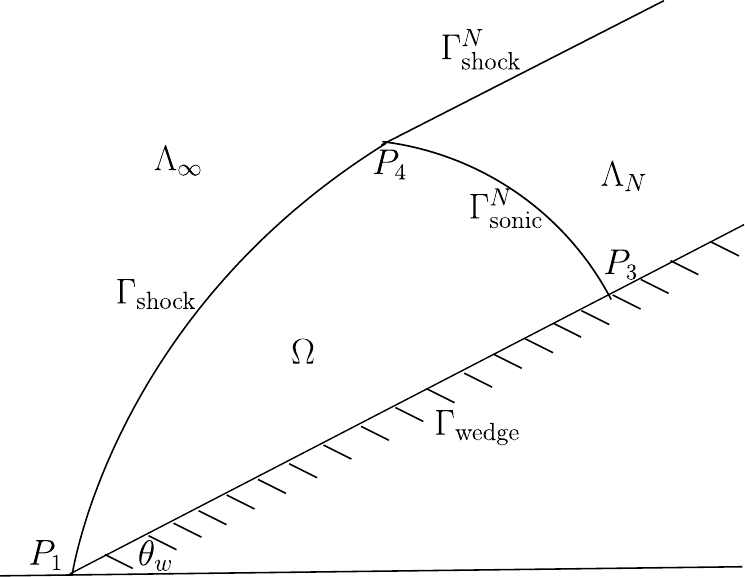}
\caption{Prandtl subsonic shock reflection}
\label{figure:Prandtl sub}
			\end{minipage}
\end{center}
\end{figure}

Consider the problem in the self-similar coordinates $\xxi=(\xi_1,\xi_2)$ in the region:
\[
\Lambda=\R^2_+\setminus\big\{\xxi:\,\xi_2>\max(0,\xi_1\tan\theta_{\rm w})\big\}.
\]
We seek global entropy solutions of the boundary value problem in the sense of Definition \ref{weakSolEulerSlipBC}.

First, a similar argument as made for the regular shock reflection at the reflection point $P_0$
(see Figs. \ref{figure:shock refelction}--\ref{figure:shock refelction-subs}) yields that, at wedge-vertex $P_1$,
for a given uniform incoming flow $(\rho_{\infty},\vv_{\infty})$, there is a detachment angle $\theta_{\rm w}^{\rm d}\in(0,\frac{\pi}{2})$
such that the system of algebraic equations \eqref{RH} and \eqref{Euler-slipBC-def} for state $(O)$ has two solutions
for each wedge-angle $\theta_{\rm w}\in(0,\theta_{\rm w}^{\rm d})$ such that the entropy condition is satisfied for the two-shock configuration.
The weak state $(O)$ with the smaller density is expected to be physical.
So we always refer to state $(O)$ as the weak state $(O)$.
Similarly, it follows from the Rankine-Hugoniot conditions \eqref{RH} and the slip boundary condition \eqref{Euler-slipBC-def} that
\begin{equation}
    \mbox{\it The shock line $\Shock^O$
    between states $(O)$ and $(\infty)$ is not parallel to $\Wedge$ for all $\theta_{\rm w}\in(0,\theta_{\rm w}^{\rm d})$.}
\end{equation}
Depending on the wedge-angle, state $(O)$ can be either supersonic or subsonic at $P_1$.
It determines the type of the reflection, supersonic or subsonic, as shown in Fig. \ref{figure:Prandtl sup} and Fig. \ref{figure:Prandtl sub}, respectively.

Second, by a straightforward computation, we know that there exists a unique constant state $(N)$,
which determines the normal reflection of state $(\infty)$ from the wedge boundary $\Wedge$ so that state $(N)$ satisfies the slip boundary
condition \eqref{Euler-slipBC-def} along $\Wedge$ and the Rankine-Hugoniot conditions \eqref{RH} along a straight line $\Shock^N$,
which lies in $\Lambda$ and is parallel to $\Wedge$.

Hinted by the solution structures given in \cite{BCFQAM2013,BCF-2,ELCPAM2008} for the potential flow,
for any given wedge-angle $\theta_{\rm w}\in(0,\theta_{\rm w}^{\rm d})$,
an entropy solution $(\rho,\vv)$ of the boundary value problem in the sense of Definition \ref{weakSolEulerSlipBC}
is called a regular Prandtl-Meyer reflection solution
for the isentropic Euler system \eqref{selfSimEuler}--\eqref{selfSimEuler-moment}
if $(\rho,\vv)$ satisfies the following further properties:
\begin{enumerate}[(i)]
\item If state $(O)$ at $P_1$ is supersonic, \emph{i.e.}, $|\vv_O(P_1)|>c_O$,
the solution has the supersonic reflection structure as in Fig. \ref{figure:Prandtl sup} at point $P_1$.
If state $(O)$ at $P_1$ is subsonic or sonic, \emph{i.e.}, $|\vv_O(P_1)|\le c_O$,
the solution has the subsonic reflection structure as in Fig. \ref{figure:Prandtl sub}.

\item The reflected shock curve $\Shock$ (\emph{i.e.}, $P_4P_5$ for the supersonic reflection and $P_1P_4$ for the subsonic reflection)
is $C^1$ up to its endpoints and is $C^2$ in its relative interior.

\item  $(\rho, \vv)$ is continuous in $\overline\Omega\cap {\mathcal N}_r(\Shock\cup\Sonic^N\cup\Sonic^O)$
when $|\vv_O(P_1)|>c_O$, and in $\overline\Omega\cap {\mathcal N}_r(\Shock\cup\Sonic^N)$
when $|\vv_O(P_1)|\le c_O$ for some $r>0$.

\item The solution coincides with states $(\infty)$, $(O)$, and $(N)$ in their respective regions.
Specifically,
\begin{equation*}
(\rho, \vv)=
\begin{cases} (\rho_{\infty}, \vv_{\infty}) \qquad\,\,&\mbox{in}\,\,\, \Lambda_{\infty},\\[0.5mm]
            (\rho_N, \vv_N) \qquad\,\,&\mbox{in}\,\,\,\Lambda_N,\\[0.5mm]
              (\rho_O, \vv_O) \qquad\,\, &\mbox{in}\,\,\,\Lambda_O\,\,\,\, \mbox{if state (O) at $P_1$ is supersonic},
\end{cases}
\end{equation*}
and
$$
\displaystyle \lim_{\xxi\in\Omega,\;\xxi\to P_1}(\rho, \vv)(\xxi)=(\rho_O, \vv_O)(P_1)\qquad\mbox{if state (O) at $P_1$ is subsonic or sonic}.
$$

\item $\vv\cdot\nu\leq -C^{-1}$ on $\overline{\Shock}\cup\overline{\Sonic^N}\cap\overline{\Sonic^O}$ for some $C>0$,
where $\vv$ on the curves is taken from the $\Omega$--side and $\nnu$ is the outer normal with respect to $\Omega$.

\item  The flow is pseudo-subsonic in $\Omega$ on and near $\Shock$, except for the sonic points $P_4$ and $P_5$ if they exist, or point $P_1$ if state $(O)$ is sonic at $P_1$.
\end{enumerate}

We remark that, for the supersonic reflections, it follows from {\rm  (ii)--(iv)} of the definition above and the Rankine-Hugoniot conditions that
$(\rho, \vv)$ is continuous across the sonic arcs $\Sonic^N$ and $\Sonic^O$.

\smallskip
Then, following the arguments as for the proof of Theorem \ref{lower-reg-Th-RegRefl},
we have the following theorem on the lower regularity of the regular Prandtl-Meyer shock reflection solutions:

\begin{theorem}\label{thm-4.11}
Let $(\rho,\vv)$ be a regular Prandtl-Meyer shock reflection solution for the isentropic Euler system
such that
\begin{enumerate}[\rm (i)]
\item $(\rho, \vv)\in C^1\big((\Nbhd_\sigma(\Shock\cup\Sonic^N\cup\Sonic^O)\cap\overline\Omega)\setminus\{P_4,\,P_5\}\big)
\cap C^{0,1}\big(\Nbhd_\sigma(\Shock\cup\Sonic^N\cup\Sonic^O)\cap\overline\Omega\big)$
for some $\sigma>0$, where $P_5$ is replaced by $P_1$ if the subsonic shock reflection occurs at $P_1${\rm ;}

\item
$|\vv|\le C_0 $ and $C_0^{-1}\le \rho\le C_0$ in $\Omega$ for some $C_0\ge 1$.
\end{enumerate}
Then $\vv\notin H^1(\Omega)$.
\end{theorem}

\begin{proof}
For the supersonic reflection, $M=3$, $\Nex=1$, and $\Nint=3$ with $\Lambda_1=\Lambda_N$, $\Lambda_2=\Lambda_{\infty}$, $\Lambda_3=\Lambda_O$,
$\GammaExt_1=\Wedge=P_2P_3$, $\GammaInt_1=\Shock=P_4P_5$,  $\GammaInt_2=\Sonic^N=P_3P_4$, and $\GammaInt_3=\Sonic^O=P_2P_5$.

For the subsonic reflection, $M=2$, $\Nex=1$, and $\Nint=2$ with $\Lambda_1=\Lambda_N$, $\Lambda_2=\Lambda_{\infty}$,
$\GammaExt_1=\Wedge=P_1P_3$, $\GammaInt_1=\Shock=P_1P_4$, and $\GammaInt_2=\Sonic^N=P_3P_4$.

Then we can follow the proof of Theorem \ref{lower-reg-Th-RegRefl} to show that the regular solutions satisfy
the assumptions of Theorem \ref{lower-reg-Th} similarly.
The only (slight) difference is in showing that
condition \eqref{PtWithNonzeroTangVelocCurv-lemma-i2}  of Lemma \ref{PtWithNonzeroTangVelocCurv-lemma} holds. We need to
show that $\Shock$ is not a straight segment.
Let us first consider the supersonic Prandtl reflection case; see Fig. \ref{figure:Prandtl sup}.
Arguing as in the proof of  Theorem \ref{lower-reg-Th-RegRefl}, we show that,
at point $P_4$ (resp. $P_5$), curve $\Shock$ is tangential to line $\Shock^N$
(resp. $\Shock^O$). If $\Shock$ is a straight segment, we obtain that
$\Shock^N$ and $\Shock^O$ lie in the same line. However, $\Shock^N$ is parallel to $\Wedge$ with a positive distance from it, while
$\Shock^O$ passes through point $P_1\in \Wedge$.
This contradiction shows that $\Shock$ is not a straight segment.
In the case of subsonic or sonic Prandtl reflection as in Fig. \ref{figure:Prandtl sub},
we argue similarly, except that we consider point $P_1$ instead of point $P_5$ and use the fact that, in the subsonic or sonic case, $\Shock$ is parallel to $\Shock^O$
at $P_1$ by the last equality in condition (iv) of the definition of admissible solutions.
Then we conclude in the same way as in the case of supersonic Prandtl reflection.
Thus, we do not repeat the similar arguments for the proof.
\end{proof}

\subsection{Lower Regularity of the Shock Diffraction Solutions of the Lighthill Problem}
The third problem is the Lighthill problem, \emph{i.e.}, the shock diffraction problem,
for the isentropic Euler system.
As discussed in \cite{ChenFeldmanHuXiang} for the potential flow and shown in Fig. \ref{figure:shock diffraction},
the Lighthill shock diffraction problem arises as a straight incident shock passes through a wedge stepping down.

\begin{figure}[!h]
  \centering
\centering
     \includegraphics[width=0.5\textwidth]{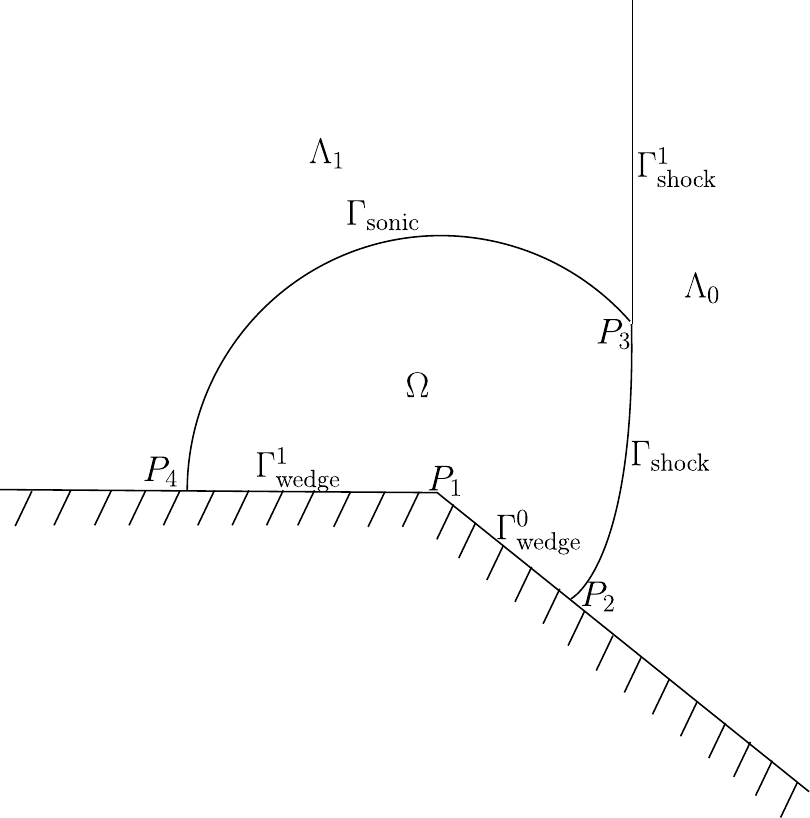}
  \caption{Lighthill shock diffraction problem}
  \label{figure:shock diffraction}
\end{figure}

Initially, we consider two piecewise constant Riemann data with the left state $(1)$: $(\rho_1,u_1^{(1)},0)$ for $u_1^{(1)}>0$
and the right state $(0)$: $(\rho_0,0,0)$, separated by a vertical shock,
which hits the wedge at the wedge-corner $P_1$.
Then the Lighthill shock diffraction problem evolves in a self-similar structure as time goes on.
In the self-similar coordinates $\xxi$, the incident shock $\Shock^1$ is given by $\xi_1=\xi_1^0$.
By a straightforward calculation, it follows from the Rankine-Hugoniot conditions \eqref{RH} and
the entropy condition \eqref{shock-entropy-RHp} that the location of the incident shock satisfies $0<\xi_1^0<c_1$,
where $c_1=\rho_1^{\frac{\gamma-1}{2}}$ is the sonic speed of state $(1)$.
Thus, as shown in Fig. \ref{figure:shock diffraction}, the incident shock $\Shock^1$ interacts with the sonic circle $\Sonic$
of state $(1)$ and becomes a transonic shock $\Shock$, and then $\Shock$ meets wedge $\Wedge^0$ perpendicularly.

Let $\Lambda_0$, $\Lambda_1$, and $\Omega$ be defined, as shown in Fig. \ref{figure:shock diffraction}.
Let $\Lambda=\overline{\Lambda_0\cup\Lambda_1\cup\Omega}$.
An entropy solution $(\rho,\vv)$ of the boundary value problem in $\Lambda$ in the sense of Definition \ref{weakSolEulerSlipBC}
is called a regular shock diffraction solution of the Lighthill problem for the isentropic Euler system
if $(\rho,\vv)$ satisfies the following further properties:
\begin{enumerate}[(i)]
\item The diffracted shock curve $\Shock=P_2P_3$ is $C^1$ up to its endpoints and is $C^2$ in its relative interior.
Its tangent is perpendicular to $\Wedge^0$ at $P_2$ and is vertical at $P_3$.

\item  $(\rho, \vv)$ is continuous in $\overline\Omega\cap {\mathcal N}_r(\Shock\cup\Sonic)$
for some $r>0$.

\item The solution coincides with states $(0)$ and $(1)$ in their respective regions.
Specifically,
\begin{equation*}
(\rho, \vv)=
\begin{cases} (\rho_0, \vv_0) \qquad\mbox{in $\Lambda_0$},\\[0.5mm]
              (\rho_1, \vv_1) \qquad \mbox{in $\Lambda_1$},
\end{cases}
\end{equation*}
where $\vv_0=-\xxi$ and $\vv_1=(u_1^{(1)},0)-\xxi$.

\item $\vv\cdot\nu\leq -C^{-1}$ on $\overline{\Shock}\cup\overline{\Sonic}$ for some $C>0$,
where $\vv$ on the curves is taken from the $\Omega$--side and $\nnu$ is the outer normal with respect to $\Omega$.

\item  The flow is pseudo-subsonic in $\Omega$ on and near $\Shock$ except for the sonic points $P_3$.
\end{enumerate}

Let $M=2$, $\Nex=2$, and $\Nint=2$ with $\Lambda_1=\Lambda_0$, $\Lambda_2=\Lambda_1$, $\GammaExt_1=\Wedge^1=P_1P_4$, $\GammaExt_2=\Wedge^0=P_1P_2$,
$\GammaInt_1=\Shock=P_2P_3$, and $\GammaInt_2=\Sonic=P_3P_4$.
Then, following the proof of Theorem \ref{lower-reg-Th-RegRefl}, we have the following theorem on the lower regularity of
the regular shock diffraction solutions.

\begin{theorem}
Let $(\rho,\vv)$ be a regular shock diffraction solution of the Lighthill problem
for the isentropic Euler system and satisfy the following{\rm :}
\begin{enumerate}[\rm (i)]
\item
$\rho, \vv\in C^1\left((\Nbhd_\sigma(\Shock\cup\Sonic)\cap\overline\Omega)\setminus\{P_3\}\right)\cap C^{0,1}\left(\Nbhd_\sigma(\Shock\cup\Sonic)\cap\overline\Omega\right)$ for some $\sigma>0${\rm ;}
\item
$|\vv|\le C_0 $ and $C_0^{-1}\le \rho\le C_0$ in $\Omega$ for some $C_0\ge 1$.
\end{enumerate}
Then $\vv\notin H^1(\Omega)$.
\end{theorem}

Since the proof argument is similar to the one for Theorem \ref{thm-4.11},
we omit the details.

\subsection{Lower Regularity of the Riemann Solutions with Four-Shock Interactions}
The final problem is the Riemann problem with four-shock interaction structure for the isentropic Euler system.
As discussed in \cite{CCHLW2023} for the Riemann problem with four-shock interactions for the potential flow,
initially in the $\xx$--coordinates, the scale-invariant domains $\Lambda_i\subset\R^2$, for $i=1,2,3,4$, are defined by
\begin{align*}
    &\Lambda_1=\big\{\xx\in\R^2\,:\,-\theta_{14}<\theta<\theta_{12}\big\},
      \qquad\qquad\quad\Lambda_2=\big\{\xx\in\R^2\,:\,\theta_{12}<\theta<\pi-\theta_{32}\big\},\\
    &\Lambda_3=\big\{\xx\in\R^2\,:\,\pi-\theta_{32}<\theta<\pi+\theta_{34}\big\},
      \qquad\, \Lambda_4=\big\{\xx\in\R^2\,:\,\pi+\theta_{34}<\theta<2\pi-\theta_{14}\big\},
\end{align*}
where $\theta$ is the polar angle of point $\xx\in\R^2$ and the four parameters $\theta_{12},\theta_{32},\theta_{34},\theta_{14}\in(0,\frac{\pi}{2})$.
On each $\Lambda_i$, suitable constant states $(i)$ with values $(\rho_i,\vv_i)$ are given for the Riemann initial data such that
any two neighboring states are connected by exactly one planar shock discontinuity, which satisfies
the Rankine-Hugoniot conditions and the entropy condition for the isentropic Euler system.
As introduced in \cite{CCHLW2023}, under the symmetry assumption that $\theta_{12}=\theta_{14}=\theta_1$, $\theta_{32}=\theta_{34}=\theta_2$,
and states in $\Lambda_2$ and $\Lambda_4$ are the same constants,
and under a structure assumption that one forward shock is generated between states $(1)$ and $(j)$ and one backward shock is generated between
states $(3)$ and $(j)$, governed by the Rankine-Hugoniot conditions \eqref{RH} and the entropy condition \eqref{shock-entropy-RHp},
we expect the Riemann problem develops a Riemann solution with the structure as shown in Fig. \ref{figure:four shocks}.

\begin{figure}[!h]
  \centering
\centering
     \includegraphics[width=0.7\textwidth]{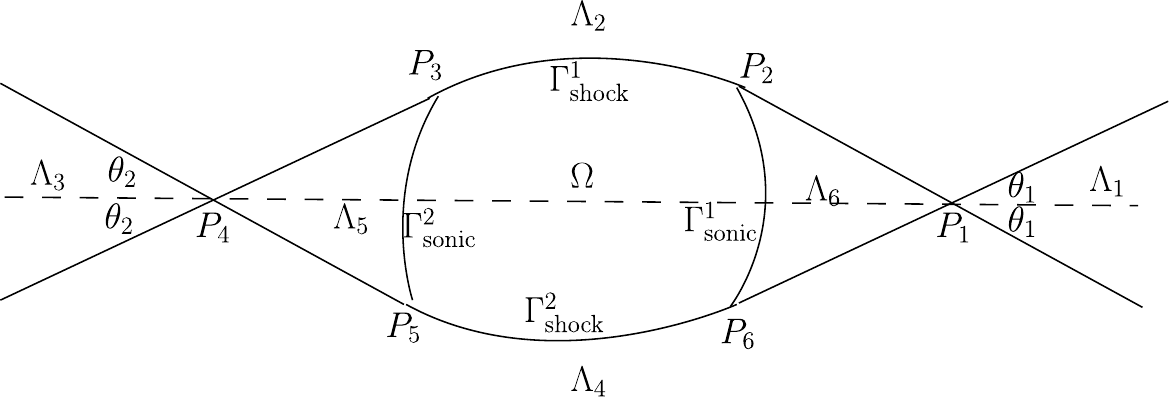}
  \caption{Riemann problem with four-shock interactions}
  \label{figure:four shocks}
\end{figure}

On the symmetry line that is the dashed line shown in Fig. \ref{figure:four shocks},
the velocity satisfies \eqref{Euler-slipBC-def} with $\bn$ being the vertical direction.
At the interaction points $P_1$ and $P_4$, based on an argument similar to the one
for the regular shock reflection problem
at the reflection point $P_0$, with the symmetry line and state $(2)$ in Fig. \ref{figure:four shocks}
corresponding to $\Wedge$ and state $(1)$ in Fig. \ref{figure:shock refelction},
there exists a detachment angle $\theta^{\rm d}\in(0,\frac{\pi}{2})$ depending on the data such that,
for $\theta_1,\theta_2\in(0,\theta_{\rm w}^{\rm d})$,
there exist two constant states $(5)$ and $(6)$ in $\Lambda_5$ and $\Lambda_6$, respectively,
satisfying the Rankine-Hugoniot conditions \eqref{RH} and the entropy condition \eqref{shock-entropy-RHp}.
One corresponds to the strong shock and the other to the weak shock.
We always select the weak one because it expects to be stable in general.
Then there exists a sonic angle $\theta^{\rm s}\in(0,\theta^{\rm d}]$ such that the state at $P_1$ (or $P_4$)
is pseudo-supersonic if $\theta_1\in(0,\theta^{\rm s})$ (or $\theta_2\in(0,\theta^{\rm s})$)
and is pseudo-subsonic if $\theta_1\in(\theta^{\rm s},\theta^{\rm d}]$ (or $\theta_2\in(\theta^{\rm s},\theta^{\rm d}]$).

When $\theta_1\in(0,\theta^{\rm s})$ at $P_1$,
there exist points $P_2$ and $P_6$ on $\Sonic^1$ (the sonic curve of state $(6)$) such that shocks $P_1P_2$ and $P_1P_6$
are straight shocks and the state in $\Lambda_6$ with boundaries $P_1P_2$, $\Sonic^1$, and $P_1P_6$
is the constant state $(6)$.
Similarly, when $\theta_2\in(0,\theta^{\rm s})$ at $P_4$,
there exist points $P_3$ and $P_5$ on $\Sonic^2$ (the sonic curve of state $(5)$) such that shocks $P_3P_4$ and $P_4P_5$
are straight shocks and the state in $\Lambda_5$ with boundaries $P_3P_4$, $\Sonic^2$, and $P_4P_5$ is the constant state $(5)$.
On the other hand, when $\theta_1\in[\theta^{\rm s},\theta^{\rm d}]$ at $P_1$ (or when $\theta_2\in[\theta^{\rm s},\theta^{\rm d}]$ at $P_4$),
the sonic arc, the straight shocks, and $\Lambda_6$ (or $\Lambda_5$) disappear near point $P_1$ (or $P_4$),
and state $(6)$ (or state $(5)$)
is the limit of solution $(\rho,\vv)$ in $\Omega$ as the point tends to $P_1$ (or $P_4$). It is direct to see that
\begin{equation}\label{tangents4.5}
    \mbox{\it The tangents to $\Shock^1$ at $P_2$ and $P_3$ are not parallel to each other},
\end{equation}
where $P_2$ is replaced by $P_1$ if state $(6)$ is subsonic and $P_3$ is replaced by $P_4$ if state $(5)$ is subsonic.

Let $\Lambda_i$ for $i=1,\cdots,6$, if they exist, and let $\Omega$ be defined as shown in Fig. \ref{figure:four shocks}.
Let $\Lambda=\R^2$.
An entropy solution $(\rho,\vv)$ of the boundary value problem in $\Lambda$ in the sense of Definition \ref{weakSolEulerSlipBC}
is called a regular Riemann solution with
four-shock interactions for the isentropic Euler system if  $(\rho,\vv)$ satisfies the following further properties:
\begin{enumerate}[\rm (i)]
\item The shock curves $\Shock^1=P_2P_3$ and $\Shock^2=P_5P_6$ are $C^1$ up to their endpoints and are $C^2$ in their relative interiors.
Their tangents satisfy \eqref{tangents4.5}.

\item  $(\rho, \vv)$ is continuous in $\overline\Omega\cap {\mathcal N}_r(\Shock^1\cup\Sonic^1\cup\Shock^2\cup\Sonic^2)$
for some $r>0$.

\item The solution coincides with the corresponding constant states in their respective regions $\Lambda_i$, for $i=1,\cdots,6$, if they exist:
Specifically,
\begin{equation*}
(\rho, \vv)=
\begin{cases} (\rho_1, \vv_1) \qquad\mbox{in}\,\,\,\Lambda_1,\\[0.5mm]
             (\rho_2, \vv_2) \qquad\mbox{in}\,\,\,\Lambda_2,\\[0.5mm]
             (\rho_3, \vv_3) \qquad\mbox{in}\,\,\,\Lambda_3,\\[0.5mm]
            (\rho_2, \vv_2) \qquad\mbox{in}\,\,\,\Lambda_4,\\[0.5mm]
              (\rho_5, \vv_5) \qquad \mbox{in}\,\,\,\Lambda_5,\quad \mbox{if state (5) at $P_4$ is supersonic}\\[0.5mm]
         (\rho_6, \vv_6) \qquad \mbox{in}\,\,\,\Lambda_6,\quad \mbox{if state (6) at $P_1$ is supersonic},
\end{cases}
\end{equation*}
where
$$
\displaystyle \lim_{\xxi\in\Omega,\;\xxi\to P_4}(\rho, \vv)(\xxi)=(\rho_5, \vv_5)(P_4)\qquad\mbox{if state (5) at $P_4$ is subsonic or sonic},
$$
$$
\displaystyle \lim_{\xxi\in\Omega,\;\xxi\to P_1}(\rho, \vv)(\xxi)=(\rho_6, \vv_6)(P_1)\qquad\mbox{if state (6) at $P_1$ is subsonic or sonic}.
$$

\item $\vv\cdot\nu\leq -C^{-1}$ on $\overline{\Shock^1}\cup\overline{\Sonic^1}\cup\overline{\Shock^2}\cup\overline{\Sonic^2}$
for some $C>0$, where $\vv$ on the curves is taken from the $\Omega$--side and $\nnu$ is the outer normal with respect to $\Omega$.

\item  The flow is pseudo-subsonic in $\Omega$ on and near $\Shock^1\cup\Shock^2$, except for possible sonic points $P_2$, $P_3$, $P_5$, and $P_6$.
\end{enumerate}

If state $(5)$ at $P_4$ and state $(6)$ at $P_1$ are both supersonic, then $M=6$, $\Nex=0$, and $\Nint=4$ with $\Lambda_i$ (for $i=1,\cdots,6$)
being given as in Fig. \ref{figure:four shocks}, $\GammaInt_1=\Shock^1=P_2P_3$, $\GammaInt_2=\Shock^2=P_5P_6$, $\GammaInt_3=\Sonic^1=P_2P_6$,
and $\GammaInt_4=\Sonic^2=P_3P_5$.

If state $(5)$ at $P_4$ and state $(6)$ at $P_1$ are both subsonic, then $M=4$, $\Nex=0$, and $\Nint=2$ with $\Lambda_i$ (for $i=1,2,3,4$)
being given as in Fig. \ref{figure:four shocks}, $\Lambda_5$ and $\Lambda_6$ disappearing, $\GammaInt_1=\Shock^1=P_1P_4$,
and $\GammaInt_2=\Shock^2=P_4P_1$.

If state $(5)$ at $P_4$ is supersonic and state $(6)$ at $P_1$ is subsonic,
then $M=5$, $\Nex=0$, and $\Nint=2$ with $\Lambda_i$ (for $i=1,\cdots,5$) being given as in Fig. \ref{figure:four shocks},
$\Lambda_6$ disappearing, $\GammaInt_1=\Shock^1=P_1P_3$, and $\GammaInt_2=\Shock^2=P_5P_1$, and $\GammaInt_3=\Sonic^2=P_3P_5$.

It state $(5)$ at $P_4$ is subsonic and state $(6)$ at $P_1$ is supersonic, then $M=5$, $\Nex=0$ and $\Nint=3$ with $\Lambda_i$ (for $i=1,\cdots,6$)
being given
as in Fig. \ref{figure:four shocks}, $\Lambda_5$ disappearing, $\GammaInt_1=\Shock^1=P_2P_4$, and $\GammaInt_2=\Shock^2=P_4P_6$, and $\GammaInt_3=\Sonic^1=P_2P_6$.

Then, following the proof of Theorem \ref{lower-reg-Th-RegRefl},
we have the following theorem on the lower regularity of the regular Riemann solution,
whose proof is omitted since it is similar to the one for Theorem \ref{lower-reg-Th-RegRefl}:

\begin{theorem}
Let $(\rho,\vv)$ be a regular Riemann solution with four-shock interactions for the isentropic Euler system and further satisfy{\rm :}
\begin{enumerate}[\rm (i)]
\item
For some $\sigma>0$,
\begin{align*}
(\rho, \vv)\in 
& \, C^1\big((\Nbhd_\sigma(\Shock^1\cup\Shock^2\cup\Sonic^1\cup\Sonic^2)\cap\overline\Omega)\setminus\{P_2,P_3,P_5,P_6\}\big)\\
& \, \cap C^{0,1}\big(\Nbhd_\sigma(\Shock^1\cup\Shock^2\cup\Sonic^1\cup\Sonic^2)\cap\overline\Omega\big),
\end{align*}
where $\Sonic^i=\emptyset$ when it does not exist for $i=1,2$;

\smallskip
\item
$|\vv|\le C_0$ and $C_0^{-1}\le \rho\le C_0$ in $\Omega$ for some $C_0\ge 1$.
\end{enumerate}
\smallskip
Then $\vv\notin H^1(\Omega)$.
\end{theorem}
%%%%%%%%%%%%%%%%%%%%%%%%%%%%%
%%%%%%%%%%%%%%%%%%%%%%%%%%%%%%%%%%%%%%%%%%%%%%%%

\smallskip
\appendix
\section{DiPerna-Lions-Type Commutator Estimates}
\label{append1}
In this appendix, we show the DiPerna-Lions-type commutator estimates, which have been used in
the proof of Theorem \ref{lower-reg-Th} in \S3.

\begin{lemma}\label{commutatorLemma}
Let $\Omega\subset\bR^n$ be open and bounded, and $n\ge 2$.
For $\eps>0$, define $\displaystyle\eta_\eps(\x)=\frac 1{\eps^n}\eta(\frac{\x}{\eps})$
for a mollifier kernel $\eta\in C^\infty_{\rm c}(\bR^n)$ satisfying $\eta(\x)\ge 0$ and $\int_{\bR^n}\eta(\x)\,\dd \x=1$.
Let $b\in W^{1,p}_{\rm loc}(\Omega)$ and $u\in  L^q_{\rm loc}(\Omega)$
for $p, q\in [1, \infty]$ such that $\displaystyle\frac 1p+\frac 1q\le 1$.
Then, for any $i=1, \cdots, n$,
\begin{equation}\label{commutatorExpr}
 \partial_{x_i}\big((bu)_\eps -b_\eps u_\eps\big)\to 0
 \qquad\mbox{in $L^1_{\rm loc}(\Omega)$ as $\eps\to 0^+$},
\end{equation}
where we have used the notation $F_\eps:=F*\eta_\eps$.
\end{lemma}

\begin{proof}
\newcommand{\subs}{\Omega'}
\newcommand{\subsT}{\Omega''}

Let $\mbox{supp}(\eta)\subset B_M$.
Denote
$$
A_\eps[u,b]:=\partial_{x_i}\big((b u)_\eps -b_\eps u_\eps\big).
$$
We divide the proof into two steps.

\smallskip
1. We first show that, for all open  $\subs\Subset\Omega$
and $\eps\in (0, \,\frac 1{2M}\dist(\subs, \partial\Omega))$,
\begin{equation}\label{estOfCommut-Main}
 \|A_\eps[u,b]\|_{L^1(\subs)}\le C(n, \eta, \subs)
 \|\grad b\|_{L^p(\Omega)}\|u\|_{L^q(\Omega)}.
 \end{equation}

Notice that
\begin{align*}
  A_\eps[u,b] & = \partial_{x_i}\big((b u)_\eps -b u_\eps\big)
  +\partial_{x_i}\big((b_\eps-b) u_\eps\big)\\
   & =\big(\partial_{x_i} (b u)_\eps-b\partial_{x_i}u_\eps \big)
   -b_{x_i} u_\eps
   +u_\eps\partial_{x_i}(b_\eps-b)
   + (b_\eps-b)\partial_{x_i}u_\eps \\
   &=:\sum_{m=1}^4 I_{\eps,m}.
\end{align*}
We now show that, for $\eps\in (0, \,\frac 1{2M}\dist(\subs, \partial\Omega))$,
\begin{equation}\label{est-commutat-I1}
\| I_{\eps,m}\|_{L^1(\subs)}\le C(n, \eta, \subs)\|\grad b\|_{L^p(\Omega)}
\|u\|_{L^q(\Omega)}\qquad \mbox{for $m=1,\cdots,4$}.
\end{equation}

First, we estimate $I_{\eps,1}=\partial_{x_i} (b u)_\eps-b\partial_{x_i}(u)_\eps$.
For $\x\in \subs$,
\begin{align*}
I_{\eps,1}(\x) & =\int_{\bR^n}\partial_{x_i}\big(\eta_\eps(\x-\y)\big)\,\big(b(\y)-b(\x)\big) u(\y)\,{\rm d}\y \\
& =\int_{B_M}\eta_{x_i}(\y)\frac{b(\x-\eps \y)-b(\x)}\eps u(\x-\eps \y)\,{\rm d}\y.
\end{align*}
Then we use the estimate of the difference quotient for a Sobolev function to obtain that,
for any $\eps\in (0, \frac 1{2M}\dist(\subs, \partial\Omega))$,
$$
\Big\| \frac{b(\cdot-\eps \y)-b(\cdot)}\eps\Big\|_{L^p(\subs)}
\le  C(n, \subs)\|\grad b\|_{L^p(\Omega)} |\y|\qquad\, \mbox{for each $\y\in B_M$}.
$$
Noting that $\|u(\cdot-\eps \y)\|_{L^q(\subs)}\le \|u\|_{L^q(\Omega)}$ for $\y$ and $\eps$ as above,
we have
$$
\|I_{\eps,1}\|_{L^1(\subs)}\le  C(n, \subs)\|\grad b\|_{L^p(\Omega)}\|u\|_{L^q(\Omega)}
\int_{B_M}|\eta_{x_i}(\y)|\,{\rm d}\y,
$$
which is \eqref{est-commutat-I1}
for $m=1$.

\smallskip
For $I_{\eps,2}=-b_{x_i} u_\eps$
and $I_{\eps,3}= u_\eps\partial_{x_i}(b_\eps-b)$, \eqref{est-commutat-I1} follows from the standard estimates.

\smallskip
Finally, we estimate $I_{\eps,4}=(b_\eps-b)\partial_{x_i}u_\eps $.
For  $\x\in \subs$,
\begin{align*}
I_{\eps,4}(\x)& =\int_{\bR^n\times\bR^n}\eta_\eps(\x-\y)\partial_{x_i}\big(\eta_\eps(\x-\zz)\big)\,
\big(b(\y)-b(\x)\big) u(\zz)\,{\rm d}\y{\rm d}\zz  \\
&=\int_{B_M\times B_M}\eta(\y)\eta_{x_i}(\zz)\frac{b(\x-\eps \y)-b(\x)}\eps u(\x-\eps \zz)\,{\rm d}\y {\rm d}\zz.
\end{align*}
Then we complete the proof of \eqref{est-commutat-I1} for
$I_{\eps,4}$ similar as for $I_{\eps,1}$ above,
by using the $L^p$--estimate
of the difference quotient for the $b$-term, and the $L^q$--estimate of the shift for the $u$-term.
Then estimate \eqref{estOfCommut-Main} holds.

\smallskip
2. Now we prove the convergence \eqref{commutatorExpr}. We first note that \eqref{commutatorExpr}
holds if $u$ has the higher regularity
$u\in W^{1,q}_{\rm loc}(\Omega)$ and $b$ as assumed.
Indeed, in this case,  $\partial_{x_i}(b u)\in L^1_{\rm loc}(\Omega)$.
Since $\Omega$ is bounded,
\begin{align*}
A_\eps[u,b]&=\big(\partial_{x_i}(b u)\big)*\eta_\eps -(\partial_{x_i}b)_\eps u_\eps-b_\eps(\partial_{x_i} u)_\eps \\
&\to \partial_{x_i}(b u)-\partial_{x_i}b u-b \partial_{x_i}u=0
\qquad\mbox{in $L^1_{\rm loc}(\Omega)\,\,$ as $\eps\to 0$}.
\end{align*}

For general $u$ and $b$ as assumed, the same convergence is obtained by approximation,
and using \eqref{estOfCommut-Main}, via a standard argument that we briefly describe now.
First, let $p\in [1,\infty)$. To include the case that $q=\infty$,
we argue as follows: Choose open $\subs\Subset\Omega$,
then choose open $\subsT$ such that $\subs\Subset\subsT\Subset\Omega$.
Let $(b_k,u_k)\in C^\infty(\Omega)$
be such that $b_k\to b$ in $W^{1,p}_{\rm loc}(\Omega)$ and
$u_k\to u$ in $L^1_{\rm loc}(\Omega)$.
Replacing $\{u_k\}$ by its subsequence
if necessary, we obtain
\begin{equation}\label{subseq-convFast}
\|\grad b_k\|_{L^\infty(\subsT)}\|u-u_k\|_{L^1(\subsT)}\to 0 \qquad\,\,\mbox{as $k\to\infty$.}
\end{equation}
Then, for each $\eps\in (0,\frac 12\dist(\subs, \partial\subsT))$,
using the bi-linearity of $A_\eps[\cdot, \cdot]$ and estimate
\eqref{estOfCommut-Main} on sets $\subs\Subset\subsT$
for $(p, q)$ and $(\infty, 1)$ respectively, we have
\begin{align*}
&\|A_\eps[u,b]\|_{L^1(\subs)} \\
& =\|A_\eps[u_k,b_k]+A_\eps[u,b-b_k]+A_\eps[u-u_k,b_k]\|_{L^1(\subs)} \\
 & \le\|A_\eps[u_k,b_k]\|_{L^1(\subs)}+C\big(
\|\grad b-\grad b_k\|_{L^p(\subsT)}\|u\|_{L^q(\subsT)}
   +\|\grad b_k\|_{L^\infty(\subsT)}\|u-u_k\|_{L^1(\subsT)}
   \big).
 \end{align*}
Recalling that $\|A_\eps[u_k,b_k]\|_{L^1(\subs)}\to 0$ as $\eps\to 0$ for each $k$
and using \eqref{subseq-convFast},
we obtain
that $\|A_\eps[u,b]\|_{L^1(\subs)}\to 0$ as $\eps\to 0$.
In the remaining case  $(p,q)=(\infty, 1)$, we argue similarly, interchanging $p$ and $q$ (resp. $b$ and $u$),
and, instead of \eqref{subseq-convFast}, we replace
$\{b_k\}$ by its subsequence
 if necessary to obtain
 $$
\|\grad b-\grad b_k\|_{L^1(\subsT)}\|u_k\|_{L^\infty(\subsT)}\to 0\qquad \mbox{as $k\to\infty$}.
 $$
Then we modify the rest of the argument correspondingly to conclude the proof.
\end{proof}

%%%%%%%%%%%%%%%%%%%%%%%%%%%%%%%%%%%%%%%%%%%%%%%%
\bigskip
\bigskip
\bigskip
\noindent{\bf Acknowledgements}.
The research of Gui-Qiang G. Chen was supported in part by the UK Engineering and Physical Sciences Research
Council Awards EP/L015811/1, EP/V008854, and EP/V051121/1.
The research of Mikhail Feldman was
supported in part by the National Science Foundation under Grants
DMS-2054689 and DMS-2219391, and the Steenbock Professorship Award from the University of Wisconsin-Madison.
The research of Wei Xiang was supported in part by the Research Grants Council of the HKSAR, China
(Project No. CityU 11304820, CityU 11300021, CityU 11311722, and CityU 11305523),
and in part by  the Hong Kong Polytechnic University internal grant (No. P0045335).
For the purpose of open access, the authors have applied a CC BY public copyright license to any Author Accepted Manuscript (AAM) version
arising from this submission.

\bigskip
\medskip
\noindent{\bf Conflict of Interest:} The authors declare that they have no conflict of interest.
The authors also declare that this manuscript has not been previously published,
and will not be submitted elsewhere before your decision.

\bigskip
\noindent{\bf Data availability:} Data sharing is not applicable to this article as no datasets were generated or analyzed during the current study.

\bigskip
\noindent{\bf Funding Declaration:} The research was supported in part by the UK Engineering and Physical Sciences Research
Council Awards EP/L015811/1, EP/V008854, and EP/V051121/1;
the US National Science Foundation under Grants
DMS-2054689 and DMS-2219391, and the Steenbock Professorship Award from the University of Wisconsin-Madison;
and the Research Grants Council of the HKSAR (China) Project Nos. CityU 11304820, CityU 11300021, CityU 11311722, and CityU 11305523,
and  the Hong Kong Polytechnic University internal grant (No. P0045335).

\bigskip
\bibliographystyle{plain}

\end{document}